\documentclass[12pt]{amsart}
\usepackage{amsmath,amssymb,enumerate,mathtools,dsfont}
\newtheorem{theorem}{Theorem}[section]
\newtheorem{claim}{}[theorem]
\newtheorem{lemma}[theorem]{Lemma}

\newtheorem{corollary}[theorem]{Corollary}
\newtheorem{conjecture}[theorem]{Conjecture}
\theoremstyle{definition}

\newcommand{\bF}{\mathbb F}

\newcommand{\bZ}{\mathbb Z}

\newcommand{\cE}{\mathcal{E}}

\newcommand{\cG}{\mathcal{G}}

\newcommand{\cM}{\mathcal{M}}

\newcommand{\cP}{\mathcal{P}}

\newcommand{\vs}[1]{\left[{#1}\right]}
\newcommand{\lvs}[1]{\vs{#1}}

\DeclareMathOperator{\cl}{cl}

\DeclareMathOperator{\PG}{PG}

\newcommand{\floor}[1]{\left\lfloor #1 \right\rfloor}
\newcommand{\ceil}[1]{\left\lceil #1 \right\rceil}

\newcommand{\del}{ \backslash  }

\numberwithin{subcase}{case}
\numberwithin{subsubcase}{subcase}
\newenvironment{subproof}[1][\proofname]{%
  \begin{proof}[Subproof:]%
}{%
  \end{proof}%
}
\newcommand{\ls}{\otimes}
\begin{document}

\title{The structure of claw-free binary matroids}
\author[Nelson]{Peter Nelson}
\address{Department of Combinatorics and Optimization, University of Waterloo, Waterloo, Canada. Email address: {\tt apnelson@uwaterloo.ca}}
\author[Nomoto]{Kazuhiro Nomoto}
\address{Department of Combinatorics and Optimization, University of Waterloo, Waterloo, Canada. Email address: {\tt 	knomoto@uwaterloo.ca}}
\thanks{This work was supported by a discovery grant from the Natural Sciences and Engineering Research Council of Canada and an Early Researcher Award from the government of Ontario}
\subjclass{05B35}
\keywords{matroids, claw-free}
\date{\today}
\begin{abstract}
	A simple binary matroid is called \emph{claw-free} if none of its rank-3 flats are independent sets. These objects can be equivalently defined as the sets $E$ of points in $\PG(n-1,2)$ for which $|E \cap P|$ is not a basis of $P$ for any plane $P$, or as the subsets $X$ of $\bF_2^n$ containing no linearly independent triple $x,y,z$ for which $x+y,y+z,x+z,x+y+z \notin X$. 
	
	We prove a decomposition theorem that exactly determines the structure of all claw-free matroids. The theorem states that claw-free matroids either belong to one of three particular basic classes of claw-free matroids, or can be constructed from these basic classes using a certain `join' operation. 
\end{abstract}

\maketitle

\section{Introduction}

This paper proves an exact structure theorem for the simple binary matroids with no three-element independent flat. All our material could be stated in terms of finite geometry, additive combinatorics, or matroid theory; the terminology we use borrows from all three areas, with an emphasis on the last.

A \emph{simple binary matroid} (hereon just a \emph{matroid}) is a pair $M = (E,G)$, where $G$ is a finite binary projective geometry $\PG(n-1,2)$, and the \emph{ground set} $E$ is any subset of the points of $G$. We abuse notation by writing $G$ for the set of points of $G$. The \emph{dimension} of $M$ is the dimension $n$ of $G$ as a geometry. Two matroids $(E_1,G_1)$ and $(E_2,G_2)$ are \emph{isomorphic} if some isomorphism from $G_1$ to $G_2$ maps $E_1$ to $E_2$. A matroid $N$ is an \emph{induced restriction} (or \emph{induced submatroid}) of $M$ if $N = (E \cap F, F)$ for some subgeometry $F$ of $G$; write $M|F$ for this matroid. If $M$ has no induced restriction isomorphic to $N$, then $M$ is \emph{$N$-free}. Our terminology is somewhat nonstandard in matroid theory; our matroids are essentially simple binary matroids in the usual sense, except we include the extrinsic ambient space $G$, which need not be spanned by the ground set, in the definition. 

If $B$ is a basis of an $n$-dimensional projective geometry $G$, we write $I_n$ for the matroid $(B,G)$. Call $I_3$ a \emph{claw}. Note that a matroid $(E,G)$ is claw-free if and only if $P \cap E$ is not a basis of $P$ for any plane $P$ of $G$. We prove that claw-free matroids can all be constructed from matroids in one of three `basic classes' of claw-free matroids via a single `join' operation that preserves the property of being claw-free. We need to define these classes and the join operation before stating the result. 

A \emph{triangle} of $G$ is a two-dimensional subgeometry. If $E$ contains no triangle of $G$ then $M = (E,G)$ is \emph{triangle-free} (a triangle-free matroid is also known as a \emph{cap} or \emph{capset}). The \emph{complement} of $M$ is the matroid $M^c = (G\del E,G)$. We say $M$ is a \emph{PG-sum} if $E$ is the disjoint union of two (possibly empty) subgeometries of $G$. Finally, call $M$ \emph{even-plane} if all of its $3$-dimensional induced restrictions have even-sized ground sets. It is easy to see that PG-sums and even-plane matroids are claw-free, and, since the complement of a claw contains a triangle, that the complements of triangle-free matroids are claw-free.

To define our join operation, it is useful to think of the points of a binary projective geometry $G$ as vectors; write $\vs{G}$ for the $n$-dimensional vector space $G \cup \{0\} \cong \bF_2^n$ associated with $G \cong \PG(n-1,2)$. Write $+ \colon \vs{G}^2 \to \vs{G}$ for the sum operation in this vector space, and extend this notation to sets, writing $X+Y$ for $\{x+y\colon x \in X, y \in Y\}$ and $x+ Y$ for $\{x\} + Y$. 

The \emph{lift-join}  of two matroids $M_1 = (E_1,G_1)$ and $M_2 = (E_2,G_2)$ is the matroid $M_1 \ls M_2 = (E,G)$, where $G$ is a projective geometry of dimension $\dim(G_1) + \dim(G_2)$ containing disjoint copies of both $G_1$ and $G_2$, and $E = E_1 \cup (E_2 + \vs{G_1})$. 

 While $\ls$ is noncommutative, we will prove in Lemma~\ref{lsassoc} that it is associative, and preserves the property of being claw-free. We can now state our main result.

\begin{theorem}\label{structure}
	A matroid $M$ is claw-free if and only if $M$ can be obtained via lift-joins from matroids that are either PG-sums, even-plane matroids, or the complements of triangle-free matroids.
\end{theorem}

This theorem resembles Chudnovsky and Seymour's classification of claw-free graphs [\ref{cp}]. The matroids $I_t$, being maximally acyclic, play the role of trees, motivating the use of the term `claw' (which describes the more symmetric of the two three-edge trees) to refer to $I_3$. Moreover, two of our basic classes are similar to natural classes of claw-free graphs; PG-sums are analogous to the graphs $G$ having a vertex $v$ adjacent to every other vertex, for which $G-v$ is the union of two cliques, and the complements of triangle-free matroids behave similarly to the complements of triangle-free graphs, which are clearly claw-free.  

 It seems that this is part of a larger phenomenon, and that our theorem forms part of the natural `exponential' analogue of the theory of the subgraph order. Binary matroids with the `restriction' order, as we have defined them, resemble simple graphs with the subgraph and induced subgraph order in a variety of deep contexts (see [\ref{bb},\ref{bkknp},\ref{fl},\ref{es},\ref{crit_threshold},\ref{green}], for example). 

Although there is a well-studied generalisation of graphs to `graphic matroids' that is vital in the theory of matroid minors, it takes less of a 
central role here, and usually the natural theorems in the `binary restriction' setting do not imply much about graphs; the techniques used to solve matroidal problems tend to   differ greatly from those for the analogous questions in graph theory, even while the graph-theoretic flavour of the theorems remains present.
This can be a blessing and a curse. The matroidal analogues of theorems about subgraphs can be very difficult, or even fail completely. However, in many other cases the matroidal proofs go more smoothly in a way that can make matroids even seem more pleasant than graphs. The main theorem of this paper is certainly one of the latter cases -- the classification of claw-free graphs required a series of long papers, and (necessarily) incorporates both technically defined basic classes and sporadic examples -- by contrast, Theorem~\ref{structure} is easy to state, and its entire proof, while not short, is contained in this paper. This gives hope that harder questions about induced restrictions are within reach, even when their graph-theoretic analogues appear outside the range of current techniques. 

In particular, one might hope for a structure theorem for the $I_t$-free matroids for larger $t$. We show later (see Corollary~\ref{claw_free_lift_join}) that such a class is closed under lift-joins, but it is not clear what other operations and basic classes might be required. 

\subsection*{$\chi$-boundedness}
Let $\chi(G)$ and $\omega(G)$ respectively denote the chromatic number and the clique number of a graph $G$. Clearly $\chi(G) \ge \omega(G)$; a class $\cG$ of graphs is \emph{$\chi$-bounded} if there is a function $f$ such that $\chi(G) \le f(\omega(G))$ for all $G \in \cG$. The class of claw-free graphs is known to have this property ([\ref{s81}]), and it is a notorious conjecture ([\ref{g85},\ref{s81}]) that the same is true when `claw-free' is replaced by `$T$-free' for any fixed tree $T$. 

For an $n$-dimensional matroid $M = (E,G)$, write $\omega(M)$ for the dimension of the largest subgeometry of $M$ contained in $E$, and define $\chi(M) = n - \omega(M^c)$. The parameter $\omega$ is a natural analogue of clique number. The parameter $\chi$, known as the \emph{critical number} of $M$ and also called \emph{critical exponent}, is an analogue of chromatic number (this may not seem obvious at first look, but turns out to be a natural analogy in many ways; for example, if $M = M(G)$ is the graphic matroid of the graph $G$, then $\chi(M) = \ceil{\log_2(\chi(G))}$ -- see [\ref{oxley}] p. 588 for a discussion). The lower bound $\chi(M) \ge \omega(M)$ is easy to show, since the dimensions of disjoint subgeometries of $G$ sum to at most $n$. As with graphs, a class $\cM$ of matroids is \emph{$\chi$-bounded} by a function $f$ if $\chi(M) \le f(\omega(M))$ for all $M \in \cM$. Unfortunately, [\ref{bkknp}] shows that the class of claw-free matroids does not enjoy this property for any function.

\begin{theorem}[{[\ref{bkknp}], Theorem 1.3}]\label{e3_unbounded}
	For all $k \ge 0$, there is an even-plane matroid $M$ with $\omega(M) \le 2$ and $\chi(M) \ge k$. 
\end{theorem}
Since even-plane matroids are claw-free, this implies that the claw-free matroids are not a $\chi$-bounded class. As a consequence of our structure theorem, we prove that even-plane matroids present the only obstruction to $\chi$-boundedness. 

\begin{theorem}\label{chi_bounded_simple}
	If $N$ is an even-plane matroid, then the class of $N$-free, claw-free matroids is $\chi$-bounded. 
\end{theorem}

\subsection*{Rough structure}


The $\chi$-boundedness in Theorem~\ref{chi_bounded_simple} follows from the fact that claw-free matroids with bounded $\omega$ are very highly structured. One way to state this structure is the following theorem. 

\begin{theorem}\label{roughclaw}
	For all $s \ge 1$ there exists $k \ge 2$ so that, if $M$ is a claw-free matroid and $\omega(M) \le s$, then $M$ is the lift-join of $2s+1$ matroids, each of which either is even-plane, or has dimension at most $k$. 
\end{theorem}

We will show (see Lemma~\ref{lj_parameters}) that $\omega(M_1 \ls M_2) = \omega(M_1) + \omega(M_2)$ for all matroids $M_1,M_2$. Since even-plane matroids have $\omega \le 2$ while $\omega(N) \le \dim(N)$ for each $N$, the conclusion of the theorem above implies that $\omega(M) \le (2s+1)k$. Thus, we can view the theorem as a qualitative structure theorem for claw-free matroids with bounded $\omega$. 

We hope that this kind of statement will apply equally when, instead of excluding a claw, we exclude $I_t$ for some arbitrary fixed $t$. That is, we want to describe the $M$ for which $\omega(M) \le s$, but also $\sigma(M) \le t$, where $\sigma(M)$ is the largest $\ell$ for which $M$ has an induced $I_{\ell}$-restriction. Ideally, we would like a theorem that describes all such matroids as belonging to some constructively described class $\cM_{s,t}$ within which $\omega$ and $\sigma$ is bounded. Theorem~\ref{roughclaw} does this when $t = 2$ (i.e. in the claw-free case); the description incorporates small matroids, even-plane matroids, and lift-joins. 

For larger $t$, the description in Theorem~\ref{roughclaw} will fail for two reasons. The first is the existence of classes that generalise the even-plane matroids. For each $t \ge 1$, let $\cE_t$ denote the class of matroids whose $t$-dimensional induced restrictions all have ground sets of even size. It was shown in [\ref{bkknp}] that $\cE_j$ properly contains $\cE_i$ for $i < j$, and that these classes are structurally rich, growing richer as $t$ grows. Theorem 3.2 of [\ref{bkknp}] gives a complete description of all these classes in terms of an iterative construction. It is easily seen that $\omega(M) < t$ for all $M \in \cE_t$, and that $\sigma(M) < t$ for all $t$ odd and $M \in \cE_t$. Therefore, any description of the matroids for which $\omega$ and $\sigma$ are bounded by $t$ must describe all matroids in $\cE_{t'}$, where $t'$ is the largest odd integer with $t' \le t$.

The second reason is because lift-joins are not general enough. For example, the \emph{direct sum} $M = M_1 \oplus M_2$ of two matroids $M_1 = (E_1,G_1)$ and $M_2 = (E_2,G_2)$ (defined as $M = (E_1 \cup E_2, G)$ where $G$ is a minimal projective geometry containing $G_1$ and $G_2$ as disjoint subgeometries) satisfies $\omega(M) = \max(\omega(M_1),\omega(M_2))$ and $\sigma(M) = \sigma(M_1)+\sigma(M_2)$; therefore, performing a bounded number of direct sums does not increase $\omega$ and $\sigma$ arbitrarily, and so (say) the class of matroids that are direct sums of at most $5$ even-plane matroids satisfies $\omega \le 2$ and $\sigma \le 10$, and is not described by any classes or constructions already discussed. 

In fact, there is a common generalisation of lift-joins and direct sums. Call a projective subgeometry of $G$ a \emph{flat} of $G$. For matroids $M_1 = (E_1,G_1)$ and $M_2 = (E_2,G_2)$ and flats $F_1,F_2$ of $G_1,G_2$ respectively, let $M_1 \ls_{F_1,F_2} M_2 = (E,G)$, where $G$, as usual, is a minimal projective geometry containing $G_1$ and $G_2$ as disjoint flats, and 
\[E= (E_1 \cup E_2) \cup (F_1  + (E_2 \cap F_2)).\]
 This is a \emph{partial lift-join}; if $F_1$ or $F_2$ is empty then it is a direct sum, and if $F_i = E_i$ for each $i$, it is a lift-join. We will show in Lemma~\ref{partial_lj} that $\sigma(M) \le 2(\sigma(M_1) + \sigma(M_2))+3$ and $\omega(M) \le \omega(M_1) + \omega(M_2)$, regardless of the $F_i$. We believe that this operation is general enough to write down a structure conjecture. 

\begin{conjecture}
	For all $s,t \ge 1$ there exists $k \ge 2$ such that, if $M$ is a matroid with $\omega(M) \le s$ and $\sigma(M) \le t$, then $M$ can be constructed via partial lift-joins from matroids $M_1, \dotsc, M_k$, where each $M_i$ is in $\cE_{k}$.
\end{conjecture}

Note, since $\sigma(M) \le k$ and $\omega(M) < k$ for all $M \in \cE_k$,  that the conclusion implies that $\sigma(M)$ and $\omega(M)$ are bounded. Note also that we do not need to distinguish the case where the $M_i$ are small, as small matroids are all in $\cE_k$ whenever $k$ is sufficiently large. We expect that if the above conjecture is true, the required $k$ will be very large compared to $s$ and $t$. 

The \emph{codimension} of a flat $F$ in a projective geometry $G$ is defined to be $\dim(G)-\dim(F)$. The following is an alternative version of the statement in Theorem~\ref{roughclaw}.

\begin{theorem}\label{flatroughomega}
	For all $s \ge 1$ there exists $k \ge 2$ such that, for every claw-free matroid $M = (E,G)$ with $\omega(M) \le s$, there is a flat $F$ of $G$ whose codimension at most $k$, such that $M|F$ is the lift-join of $2s+1$ even-plane matroids. 
\end{theorem}

It is clear that $\omega(M) \le \omega(M|F) + k$ for every codimension-$k$ flat, so the outcome of this theorem still certifies bounded $\omega$. For each matroid $M = (E,G)$, define $\alpha(M) = \omega(M^c) = \dim(M) - \chi(M^c)$. This is the analogue of the `independence number' of a graph, and is equal to the largest dimension of an induced submatroid of $M$ whose ground set is empty. Another consequence of Theorem~\ref{structure} is a qualitative structure theorem where we bound $\alpha$ instead of $\omega$ in the hypothesis. 

\begin{theorem}\label{flatroughalpha}
	For all $s \ge 1$ there exists $k \ge 2$ such that, for every claw-free matroid $M = (E,G)$ with $\alpha(M) \le s$, there is a flat $F$ of $G$ whose codimension at most $k$, such that $M|F$ is the lift-join of $2s+1$ matroids whose complements are triangle-free. 
\end{theorem}

\subsection*{Density}

	We say a matroid $M = (E,G)$ is \emph{full-rank} if $G = \cl(E)$. We prove a theorem that exactly determines the sparsest claw-free matroids, showing that they are exponentially dense. 
	
	\begin{theorem}\label{density}
		Let $M = (E,G)$ be a full-rank, $r$-dimensional claw-free matroid. Then $|E| \ge 2^{\lfloor r/2 \rfloor} + 2^{\lceil r/2 \rceil} - 2$. Equality holds precisely when $E$ is the disjoint union of two flats of $G$ of dimensions $\lfloor r/2 \rfloor$ and $\lceil r/2 \rceil$. 
	\end{theorem}

	It is a rare property for $N$ to give rise to an exponential lower bound like the above on the density of full-rank $N$-free matroids; indeed, if $N \not\cong I_t$ is a full-rank $t$-dimensional matroid, then the matroid $I_r$ is a full-rank $N$-free matroid with only linearly many elements. We conjecture that Theorem~\ref{density} will generalise to excluding an arbitrary $I_t$. 
	
	\begin{conjecture}
		Let $t \ge 2$ and $r \ge t$ be integers. If $M = (E,G)$ is a full-rank, $r$-dimensional matroid with no induced $I_t$-restriction, then \[|E| \ge s 2^{\ceil{r/(t-1)}} + (t-1-s) 2^{\floor{r/(t-1)}}- (t-1),\] 
		where $s$ is the remainder of $r$ on division by $t-1$. Equality holds precisely when $E$ is the disjoint union of flats $F_1, \dotsc, F_{t-1}$ whose dimensions sum to $r$, where $|\dim(F_i) - \dim(F_j)| \le 1$ for all $i,j$. 
	\end{conjecture}
	Although this extremal structure is simple, we expect that this conjecture will be difficult to prove. Even for $t = 3$, we do not know a way that does not use the full strength of Theorem~\ref{structure}. The following weaker conjecture might be more amenable to a direct proof. Theorem~\ref{density} implies that $\lambda = \sqrt{2}$ satisfies the conjecture when $t = 3$. 
	
	\begin{conjecture}
		For all $t \ge 3$ there exists $\lambda > 1$ such that for each $r \ge t$, every full-rank, $r$-dimensional matroid $M = (E,G)$ with no induced $I_t$-restriction satisfies $|E| > \lambda^r$. 
	\end{conjecture}
	
\subsection*{Excluding anticlaws}

A \emph{anticlaw} is the complement of a claw. Using our structure theorem for claw-free matroids, it is possible to give a structure theorem for claw-free, anticlaw-free matroids. We say that a matroid $M = (E,G)$ is a \emph{target} if there exist flats $F_0 \subseteq F_1 \subseteq \dotsc \subseteq F_k \subseteq G$ such that $E$ is the union of $F_{i+1} \del F_i$ for all even $i$ (note that flats are allowed to be empty).

	\begin{theorem}\label{anticlaw}
		A matroid $M$ is claw-free and anticlaw-free if and only if $M$ is a target.
	\end{theorem}

\section{Preliminaries}

\subsection*{Flats}

We call a subgeometry of $G$ a \emph{flat} of $G$. Flats correspond to vector subspaces, and hence a flat $F$ is equivalently characterised by the property that $F \cup \{0\}$ is closed under addition; recall that we write $\vs{F}$ for this set. We call flats of dimension $2$ and $3$ \emph{triangles} and \emph{planes} respectively. (The word `line' may seem more appropriate for a $2$-dimensional flat, but we go with `triangle' in analogy with graph theory.) A maximal proper flat of $G$ is a \emph{hyperplane}. A $d$-dimensional flat has $2^d-1$ elements.

Note that a triangle of $G$ is equivalently a triple $\{x,y,x+y\} \subseteq G$ with $x \ne y$. As well as $\vs{F}$ being closed under addition, it is also clearly true that $x+y \in F$ for all distinct $x,y \in F$; i.e. no triangle of $G$ contains exactly two elements of $F$. Taking complements, we see that if no triangle of $G$ contains exactly one element of $E$, then $G\del E$ is a flat of $G$. Matroids with this property play a special role in extremal theory. For an integer $t \ge 0$, a \emph{Bose-Burton geometry of order $t$} is a matroid $M = (E,G)$ for which $G\del E$ is a flat of dimension $\dim(G)-t$. These matroids are named after the authors of [\ref{bb}], who proved the following geometric analogue of Tur\'an's theorem.

\begin{theorem}\label{bbt}
	Let $M = (E,G)$ be a matroid. If $t \ge 0$ and $E$ contains no $(t+1)$-dimensional flat of $G$, then $|E| \le 2^{\dim(M)}(1-2^{-t})$. If equality holds, then $M$ is an order-$t$ Bose-Burton geometry.
\end{theorem}

If we know not just that every triangle $T$ satisfies $|T \cap F| \ne 2$, but that $|T \cap F|$ is always odd, then $F$ must be either equal to $G$ or a hyperplane of $G$, since otherwise $F$ is a flat of codimension at least $2$, so its complement contains a triangle.

\subsection*{Lift-joins}

A \emph{coset} of a flat $F$ in a flat $G \supseteq F$ is any proper translate of the subspace $\vs{F}$ of $\vs{G}$, i.e. a set of the form $F = x + \vs{F}$ for some $x \in G \del F$. The set $G \del F$ partitions into cosets of $F$. We do not consider the set $F$ itself a coset. 

For disjoint projective geometries $G_1,G_2$, write $G_1 \oplus G_2$ for the projective geometry $G = (\vs{G_1} \oplus \vs{G_2}) \del \{0\}$, where the second `$\oplus$' denotes the vector space direct sum. Note that $\dim(G) = \dim(G_1) + \dim(G_2)$; we naturally identify each $\vs{G_i}$ with its copy in $\vs{G_1} \oplus \vs{G_2}$, and accordingly think of $G_1$ and $G_2$ as disjoint subgeometries of $G_1 \oplus G_2$. We can thus define the lift-join of matroids $M_1 = (E_1,G_1)$ and $M_2 = (E_2,G_2)$ slightly more formally by \[M_1 \ls M_2 = (E_1 \cup (E_2 + \vs{G_1}),G_1 \oplus G_2).\]  Note that $M_i = (M_1 \ls M_2)|G_i$ for each $i \in \{1,2\}$. A set $Y \subseteq G$ is \emph{mixed} with respect to a matroid $M = (E,G)$ if $Y$ intersects both $E$ and $G\del E$; otherwise it is \emph{unmixed}.  We now prove a lemma that will allow us to recognise lift-joins.

\begin{lemma}\label{rlj}
	Let $M = (E,G)$ be a matroid and $F$ be a flat of $G$. The following are equivalent. 
	\begin{enumerate}
		\item\label{rlj1} $F$ has no mixed cosets with respect to $M$. 
		\item\label{rlj2} $M$ is a lift-join of $M|F$ and another induced submatroid of $M$.
		\item\label{rlj3} $M$ is the lift-join of $M|F$ and $M|J$ for each maximal flat $J$ of $G$ that is disjoint from $F$.
	\end{enumerate}
\end{lemma}
\begin{proof}
	Suppose that (\ref{rlj1}) holds. Let $J$ be a maximal flat of $G$ that is disjoint from $F$; note that each coset of $F$ intersects $J$ in exactly one element. For each $x \in J$, the fact that the coset $\vs{F}+x$ is unmixed implies that $(\vs{F}+x) \subseteq E$ if and only if $x \in E$, and $(\vs{F}+x) \cap E = \varnothing$ if and only if $x \notin E$. Therefore $E = (E \cap F) \cup \bigcup_{x \in E\cap J}(\vs{F}+x) = (E \cap F) \cup ((E \cap J) + \vs{F})$. It follows that $M = (M|F) \ls (M|J)$. Thus (\ref{rlj1}) implies (\ref{rlj3}).
	
	Clearly (\ref{rlj3}) implies (\ref{rlj2}); suppose that (\ref{rlj2}) holds, so $M$ is the lift-join of $M|F$ and $M|K$ for some flat $K$ of $G$ that is disjoint from $F$. Thus $E = (E \cap F) \cup ((E \cap K)+\vs{F})$. Let $x,y \in G \del F$ belong to the same coset of $F$ in $G$, so $x+y \in F$. If $x \in E$ then $x \in (E \cap K) + \vs{F}$ and so $y \in (E \cap K) + \vs{F} \subseteq E$. Therefore $x \in E$ implies that $y \in E$; it follows that $F$ has no mixed cosets. 
\end{proof}

If $M = (E,G) $ is a matroid and $F$ is a nonempty proper flat of $G$ that has no mixed cosets with respect to $M$, then we call $F$ a \emph{decomposer} of $M$, and say that \emph{$F$ decomposes $M$}. The above lemma shows that $M$ is a lift-join of two smaller matroids if and only if $M$ has a decomposer. 

Note that if $F$ decomposes $M$ and $F'$ decomposes $M|F$, then $F'$ also decomposes $M$, since each coset of $F$ partitions into cosets of $F'$. 

\begin{lemma}\label{lsassoc}
	$\ls$ is associative.
\end{lemma}
\begin{proof}
	Let $M_i = (E_i,G_i)$ for $i \in \{1,2,3\}$ with the $G_i$ disjoint. Both the matroids $(M_1 \ls M_2) \ls M_3$ and $M_1 \ls (M_2 \ls M_3)$ have ambient space $G = G_1 \oplus G_2 \oplus G_3$. The first has ground set 
	\[
		(E_1 \cup (E_2 + \lvs{G_1})) \cup (E_3 + \lvs{G_1 \oplus G_2}) = E_1 \cup (E_2 + \lvs{G_1}) \cup (E_3 + \lvs{G_1} + \lvs{G_2}).
	\]
	The second has ground set 
	\[
	E_1 \cup ((E_2 \cup (E_3 + \vs{G_2}))+\vs{G_1}) = E_1 \cup (E_2 + \vs{G_1}) \cup (E_3 + \vs{G_1} + \vs{G_2}),
	\]
	giving the lemma. 
\end{proof}

To show that lift-joins preserve the property of being claw-free, we prove something more general. Recall that $M = (E,G)$ is full-rank if $G = \cl(E)$.

\begin{lemma}
 	Let $N$ be a full-rank matroid of dimension at least $3$, containing no four distinct elements that sum to zero. If $M_1$ and $M_2$ are $N$-free matroids, then so is $M_1 \ls M_2$. 
\end{lemma}
\begin{proof}
	Let $M_i = (E_i,G_i)$ for $i \in \{1,2\}$ and let $M = M_1 \ls M_2 = (E,G)$. By Lemma~\ref{rlj}, no coset of $G_1$ is mixed with respect to $M$. Let $F$ be a flat of $G$ for which $M|F \cong N$. Since $M_1$ is $N$-free, we have $F \not\subseteq G_1$; since $\cl(E \cap F)$ has dimension $\dim(F) > \dim(F \cap G_1)$, we have $(E \cap F)  \del G_1 \ne \varnothing$; let $x \in (E \cap F) \del G_1$. 
	
	If $\dim(F \cap G_1) \ge 2$ then let $v,w \in F \cap G_1$ be distinct. Since $\vs{G_1}+x$ is unmixed and intersects $E$, we have \[E \supset \vs{G_1} + x \supset \{x,x+w,x+v,x+v+w\}.\] These four distinct elements all belong to $E \cap F$; this is a contradiction, since they sum to zero. Thus $\dim(F \cap G_1) \le 1$. 
	
	If $|F \cap G_1| = 1$ then let $w$ be its element; since $\dim(E \cap F) \ge 3$ there is some $y \in (F \cap E) \del \cl(\{x,w\})$. Now $y \in E \del G_1$ and so $\vs{G_1}+ y\subseteq E$; it follows that $\{x,x+w\} \subseteq E$ and $\{y,y+w\} \subseteq E$. But $x,x+w,y,y+w$ are distinct elements of $F \cap E$ with sum zero, again a contradiction. 
	
	If $F \cap G_1 = \varnothing$ then, for each element $z$ of $F$, the coset $\vs{G_1}+z$ intersects $F$ in exactly one element $z'$ and moreover, the map $\psi \colon z \mapsto z'$ is a linear injection from $F$ to $G_2$. Since $G_1$ has no mixed cosets, we have $\psi(z) \in E$ if and only if $z \in E$, and so $M|F$ and $M|\psi(F)$ are isomorphic. The latter is an induced restriction of the $N$-free matroid $M|G_2$, giving a contradiction to $M|F \cong N$.  
\end{proof}

For $t \ge 3$, the matroid $I_t$ satisfies the hypotheses of the above lemma. This gives the following. 

\begin{corollary}\label{claw_free_lift_join}
	If $t \ge 3$ is an integer then the class of $I_t$-free matroids is closed under lift-joins. 
\end{corollary}

The next two lemmas establish some more properties of lift-joins. 

\begin{lemma}\label{lift_join_iso}
	If $F$ is a decomposer of a matroid $M = (E,G)$, and $K,K'$ are flats of $G$ disjoint from $F$ for which $\cl(F \cup K) = \cl(F \cup K')$, then there is an isomorphism $\psi$ from $M|K$ to $M|K'$ for which each $x \in K$ satisfies $x + \psi(x) \in \vs{F}$. 
\end{lemma}
\begin{proof}
	Since $K$ and $K'$ are each disjoint from $F$ while $\cl(F \cup K) = \cl(F \cup K')$, we have $\dim(K) = \dim(K')$, and each coset of $F$ in $\cl(F \cup K)$ meets $K$ and $K'$ in a single element each. For each $x \in K$, let $\psi(x)$ be the unique element of $(\vs{F}+x) \cap K'$. Since $F$ is a decomposer we have $\psi(x) \in E$ if and only if $x \in E$. For each $x
	 \in K'$ the coset $\vs{F}+x'$ intersects $K$ in some point $x$ for which $x' = \psi(x)$, so $\psi$ is surjective and thus bijective. Finally, for distinct $x,y \in K$ the elements $\psi(x)+\psi(y)$ and $\psi(x+y)$ are both in $(\vs{F}+(x+y)) \cap K'$, so are equal. Thus $\psi$ is an isomorphism.
\end{proof}

\begin{lemma}\label{lj_parameters}
	If $M \cong M_1\ls M_2$, where $M_i = (E_i,G_i)$ then 
	\begin{itemize}
		\item $M^c \cong M_1^c \ls M_2^c$,
		\item $M|\cl(F_1 \cup F_2) = (M_1|F_1) \ls (M_2|F_2)$ for all flats $F_1,F_2$ of $G_1,G_2$,
		\item $\omega(M) = \omega(M_1) + \omega(M_2)$, and
		\item $\chi(M) = \chi(M_1) + \chi(M_2)$.
	\end{itemize}
\end{lemma}
\begin{proof}
	Let $M = (E,G)$, where $G = G_1 \oplus G_2$ and $E = E_1 \cup (\vs{G_1} + E_2)$. To see the first part, note that 
	\begin{align*}
		E^c &= G \del (E_1 \cup (\vs{G_1} + E_2)) \\
			&= (G_1 \del E_1) \cup (\vs{G_1} + G_2) \del (\vs{G_1}+ E_2)\\
			&= (G_1 \del E_1) \cup (\vs{G_1} + (G_2 \del E_2)),
	\end{align*}
	where we use the fact that every $z \in \vs{G_1} + G_2$ is uniquely expressible in the form $z = x+y$ for $x \in \vs{G_1}$ and $y \in G_2$. This gives $M^c = M_1^c \ls M_2^c$. 
	
	
	For the second part, let $F = \cl{(F_1 \cup F_2)}$. We have 
	\begin{align*}
		E \cap F &= ((E_1 \cup E_2) \cap F) \cup ((F_1 + (E_2 \cap F_2)) \cap F)\\
				 &= (E_1 \cap F_1) \cup (E_2 \cap F_2) \cup (F_1 + (E_2 \cap F_2)) \\
				 &= (E_1 \cap F_1) \cup (\vs{F_1} + (E_2 \cap F_2)),
	\end{align*}
	from which it follows that $M|F = (M_1|F_1) \ls (M_2|F_2)$.
		
	 For each $i \in \{1,2\}$ let $\omega_i = \omega(M_i)$ and let $K_i \subseteq E_i$ be a flat of dimension $\omega_i$. Then $K_1 \cup (\vs{K_1} + K_2) = \cl(K_1 \cup K_2)$ is a flat of $G$ contained in $E$ by the previous part, giving $\omega(M) \ge \omega_1+\omega_2$. 
	
	Now let $F$ be a flat of $G$ for which $F \subseteq E$. Let $K$ be a maximal flat of $F$ that is disjoint from $G_1$; note that $\dim(F) = \dim(K) + \dim(F \cap G_1) \le \dim(K) + \omega_1$. Since $G = G_1 \oplus G_2$, the flat $\cl(G_1 \cup K)$ intersects $G_2$ in a flat $K'$ for which $\cl(G_1 \cup K) = \cl(G_1 \cup K')$. By Lemma~\ref{lift_join_iso} we have $M|K \cong M|K' = M_2|K'$, and since $K \subseteq E$ this implies that $K' \subseteq E_2$ and so $\dim(K') \le \omega_2$. Thus $\dim(F) \le \omega_1+\omega_2$, and the third part follows. 
	Finally, let $n = \dim(G)$ and $n_i = \dim(G_i)$ for each $i$. We have
	\[
		\chi(M) = n - \omega(M^c) 
		= (n_1+n_2) - \omega(M_1^c) - \omega(M_2^c) 
		= \chi(M_1) + \chi(M_2)
	\] 
	giving the last part.
\end{proof}

Two nice special cases of lift-joins are when either $M_1$ or $M_2$ has dimension $1$, or equivalently when $M$ has a decomposer that is either a single element or a hyperplane. These cases correspond to more elementary constructions that preserve being claw-free. Since they are important in our proof, we discuss them here. Let $M = (E,G)$.

If $H$ is a hyperplane of $G$, it is easy to see that $H$ decomposes $M$ if and only if $G \del H$ is either contained in $E$ or disjoint from $E$. In the latter case we have $\cl(E) \subseteq H$, so $M$ is not full-rank; conversely, if $M$ is not full-rank, then any hyperplane containing $\cl(E)$ is a decomposer.

If an element $a$ of $G$ satisfies $a + E = E$, then $a \notin E$ and $\{a\}$ is a decomposer of $M$. If $a$ satisfies $a + (E \cup \{0\}) = E \cup \{0\}$, then $a \in E$ and $\{a\}$ is a decomposer of $M$. In fact, it is easy to see that $\{a\}$ is a decomposer of $M$ if and only if one of these two statements holds. If $a + E = E$, then by Lemma~\ref{lift_join_iso}, all restrictions of $M$ to hyperplanes not containing $a$ are isomorphic; for each hyperplane $H$ of $G$ with $a \notin H$, we say that $M$ is a \emph{doubling} of $M|H$. So every matroid that is a doubling has a decomposer.

We now prove the fact asserted in the introduction about the interaction of the parameters $\omega$ and $\sigma$ with partial lift-joins. This lemma will not be needed in the rest of the paper. Recall that a partial lift-join of $M_1 = (E_1,G_1)$ and $M_2 = (E_2,G_2)$ is a matroid $M = M_1 \ls_{F_1,F_2} M_2 =  (E,G_1 \oplus G_2)$, where $E = (E_1 \cup E_2) \cup (F_1 + (E_2 \cap F_2))$ for some flats $F_1,F_2$ of $G_1,G_2$ respectively. Recall also that $\sigma(M)$ is the largest $\ell$ for which $M$ has an induced $I_\ell$-restriction. 

\begin{lemma}\label{partial_lj}
	If $M$ is a partial lift-join of matroids $M_1$ and $M_2$, then 
	\begin{itemize}
		\item $\omega(M) \le \omega(M_1) + \omega(M_2)$ and
		\item $\sigma(M) \le \max(3, 2(\sigma(M_1)+\sigma(M_2)))$
	\end{itemize}
\end{lemma}
\begin{proof}
	Let $M = (E,G)$ and $M_i = (E_i,G_i)$ for each $i$, where $G = G_1 \oplus G_2$ and $E = (E_1 \cup E_2) \cup (F_1 + (E_2 \cap F_2))$. Let $E' = E_1 \cup (\vs{G_1} + E_2)$, so $(E',G) = M_1 \ls M_2$. It is clear that $E \subseteq E'$, from which it follows that $\omega(M) \le \omega(M_1 \ls M_2) = \omega(M_1) + \omega(M_2)$ by Lemma~\ref{lj_parameters}.	
		
	Let $t = \sigma(M)$ and consider an induced $I_t$-restriction $M|K$ of $M$. Let $J = E \cap K$. Let $F = \cl(F_1 \cup F_2)$ and $(J_1,J_2) = (J \cap F, J \del F)$. From Lemma~\ref{lj_parameters} it follows that $M|F = (M_1|F_1) \ls (M_2|F_2)$. Now $M|\cl(J_1)$ is an induced $I_{|J_1|}$-restriction of $M|F$, and so $|J_1| \le \sigma(M|F)$. We now claim that $\sigma(M | F) \le \max{(2, \sigma(M_1 | F_1), \sigma(M_2 | F_2))}$. Let $s = 
	\sigma(M|F)$. If $s > \max{( \sigma(M_1 | F_1), \sigma(M_2 | F_2))}$ then $M_1 | F_1$ and $M_2 | F_2$ are $I_s$-free but $M|F$ is not; since $M|F = (M_1|F_1) \ls (M_2|F_2)$, Corollary~\ref{claw_free_lift_join} gives a contradiction unless $s \le 2$; this implies the claimed bound.
		
	The matroid $M|\cl(J_2)$ is an induced $I_{|J_2|}$-restriction of $M$. Since $F$ is a flat and $J$ is an independent set, we also have $F \cap \cl(J_2) = \varnothing$. Therefore $J_2 \subseteq E \del F \subseteq E_1 \cup E_2 \subseteq G_1 \cup G_2$. For each $i \in \{1,2\}$, we have $\cl(J_2 \cap G_i) \cap F_i = \varnothing$, so $M_i|\cl(J_2 \cap G_i) = M|\cl(J_2 \cap G_i) \cong I_{|J_2 \cap G_i|}$. Therefore $|J_2 \cap G_i| \le \sigma(M_i)$ for each $i \in \{1,2\}$, and so $|J_2| = |J_2 \cap (G_1 \cup G_2)| \le \sigma(M_1) + \sigma(M_2)$.  Combining the above bounds, we get 
	\[\sigma(M) = |J| \le  \max{(2, \sigma(M_1|F_1), \sigma(M_2|F_2))}  + \sigma(M_1) + \sigma(M_2),\]
	which implies that $\sigma(M) \le \max(3, 2(\sigma(M_1)+\sigma(M_2)))$ as required.
\end{proof}

\subsection*{PG-sums}

Recall that a matroid $M = (E,G)$ is a PG-sum if $E$ is the disjoint union of at most two flats of $G$. We say that $M$ is a \emph{strict PG-sum} if $M$ is full-rank and its ground set is the union of exactly two disjoint nonempty flats $F_1,F_2$. If this is the case, then $G \del E = F_1 + F_2$, and in fact $|F_1+F_2| = |F_1||F_2|$; i.e. every $x \in G \del E$ is uniquely expressible as $x = x_1 + x_2$ where $x_1 \in F_1$ and $x_2 \in F_2$. 
The next lemma shows that PG-sums are `perfect'.

\begin{lemma}\label{pg_sum_chi}
	If $M$ is a PG-sum then $\chi(M) = \omega(M)$. 
\end{lemma}
\begin{proof}
	It suffices to consider the case where $M = (E,G)$ is a strict PG-sum, since otherwise either $M$ is not full-rank and we can pass to the restriction $(E,\cl(E))$, or $E = G$ and the conclusion is obvious. Let $F_1,F_2$ be disjoint flats of $G$ whose union is $E$, and let $n_i = \dim(F_i)$ and $n = \dim(M)$, so $n = n_1 + n_2$. Assume that $n_1 \le n_2$; we clearly have $\omega(M) = n_2$. 
	
	Let $F_1' \subseteq F_2$ be an $n_1$-dimensional flat and let $\psi$ be an isomorphism from $M|F_1$ to $M|F_1'$. Let $K = \{x+\psi(x)\colon x \in F_1\}$. Note that $K \subseteq F_1 + F_2 = G \del E$. Since $F_1$ and $F_1'$ are flats, for distinct $x,y \in K$ we have $x+y \in K$. Finally, since $|F_1||F_2| = |F_1+F_2|$, we have $|K| = |F_1|$ so $\dim(K) = n_1 = n-n_2$. It follows that $\chi(M) \le n_2 = \omega(M)$, and the result follows. 
\end{proof}

The class of PG-sums is clearly closed under taking induced restrictions; it follows from this that PG-sums are claw-free. In fact, we can easily characterize the class by forbidding four particular three-dimensional induced restrictions. Define $P_5$ and $K_4$ to be the unique three-dimensional matroids on $5$ and $6$ elements respectively (the latter's name comes from its association with the complete graph on four vertices). Let $C_4$ denote the four-element three-dimensional matroid whose ground set sums to zero. 

\begin{lemma}\label{PG_characterisation}
$M$ is a PG-sum if and only if it is $(I_3,C_4,P_5,K_4)$-free. 
\end{lemma}

\begin{proof}
The forwards direction follows from the fact that the class of PG-sums is closed under taking induced restrictions, and that the matroids $I_3$, $C_4$, $P_5$ or $K_4$ are not PG-sums. 

Conversely, suppose that $M = (E,G)$ is $(I_3,C_4,P_5,K_4)$-free. Let $F$ be a largest flat of $G$ that is contained in $E$. We may assume that $3 \le |E| < |G|$ and that $F \ne E$, as otherwise $M$ is a PG-sum. 

Suppose first that $|F|=1$. Let $v_1 \in F$ and $v_2, v_3 \in E \del F$. By maximality, $E$ contains none of the elements $v_1+v_2, v_1+v_3, v_2+v_3$. But then $M |\cl(\{v_1, v_2, v_3\})$ is isomorphic to either $I_3$ or $C_4$, a contradiction. So $|F| \geq 3$. Let $v_1 \in E \del F$. Then for any $v_2 \in F$, we must have $v_1 + v_2 \notin E$; otherwise, by maximality there exists $v_3 \in F$ such that $v_1 + v_3 \notin E$, but then $M | \cl(\{v_1, v_2, v_3\})$ has either $5$ or $6$ elements, so is either an induced $P_5$ or $K_4$-restriction. Now, let $v_1, v_2 \in E \del F$. We claim that $v_1 + v_2 \in E \del  F$, which will imply that $E \del F$ is a flat of $G$ and hence that $M$ is a PG-sum. Suppose not, and let $v_3 \in F$. Then $M | \cl( \{v_1, v_2, v_3\} )$ is isomorphic to either $I_3$ or $C_4$, a contradiction. 
\end{proof}

\subsection*{Even-plane matroids}

Write $\cE_3$ for the class of even-plane matroids (the matroids $(E,G)$ where $|E \cap P|$ is even for every plane $P$ of $G$). The material we need on even-plane matroids comprises a pair of results from [\ref{bkknp}], which considers their internal structure in detail, and a straightforward lemma. The first is easy to state.

\begin{theorem}[{[\ref{bkknp}], Theorem~1.2}]\label{even_plane_bounded_cn}
	If $M,N \in \cE_3$ and $M$ has no $N$-restriction, then $\chi(M) \leq \dim(N)+4$. 
\end{theorem}

The second result we need from [\ref{bkknp}] requires a definition. We say that a matroid $M = (E,G)$ is a \emph{semidoubling} of a matroid $N$ if there is a hyperplane $H$ of $G$, a hyperplane $H'$ of $H$, and an element $a$ of $G \del (H \cup E)$ for which $M|H = N$, while $E = (E \cap H) \cup (a + ((H \del E) \Delta H'))$. 
This condition is equivalent to the statement that for each $x \in H$, we have $a + x \in E$ if and only if either $x \in H' \cap E$ or $x \in (H\del H')\del E$. 

The significance of semidoublings is that they preserve the property of being even-plane. 

\begin{theorem}[{[\ref{bkknp}, Corollary 3.4]}]\label{even_plane_doubling_semidoubling}
	The class of even-plane matroids is closed under doublings and under semidoublings.
\end{theorem}

Finally, the following lemma is immediate.

\begin{lemma}\label{sym_diff_even_plane}
If $M=(E,G) \in \cE_3$ and $H$ is a hyperplane of $G$, then $(E \Delta (G \del H), G) \in \cE_3$.
\end{lemma}

\begin{proof}
Let $P$ be a plane of $G$. Note that $|P \cap (G \del H)|$ is either $0$ or $4$, and that $|E \cap P|$ is even. Hence $|E \Delta (G \del H) \cap P| = |E \cap (P \cap H)|+|(P \cap (G \del H)) \del E| \equiv |E \cap (P \cap H)| + |P \cap (G \del H) \cap E| = |E \cap P|\pmod{2}$.
\end{proof}

\subsection*{Restricted Triangles}

Finally, we state and prove a lemma that gives a global structure in a matroid for which certain types of triangle are forbidden. We use this result remarkably often to find decomposers. 

\begin{lemma}[Coset Lemma]\label{PQR}
	Let $(P, Q, R)$ be a partition of a binary projective geometry $G$ for which no triangle $T$ of $G$ satisfies $|T \cap P| \ge 1$ and $|T \cap R| = 1$. 
	Then
	\begin{itemize}
		\item $\cl(P) \subseteq P \cup Q$, and
		\item All cosets of $\cl(P)$ in $G$ are contained in $Q$ or $R$.
	\end{itemize}
	Furthermore, if $(R_1,R_2)$ is a partition of $R$ and $G$ has no triangle that intersects $P$, $R_1$ and $R_2$, then all cosets of $\cl(P)$ in $G$ are contained in $Q$, $R_1$ or $R_2$. 
\end{lemma}

\begin{proof}
Let $0X = \{0\}$ for all $X \subseteq \vs{G}$, and let $kX = X+ (k-1)X$ for all $k \ge 1$. Note that $\vs{\cl(X)} = \cup_{k \ge 1}(kX)$. Let $P' = P \cup \{0\}$. The triangle condition given implies that $P' + (P' \cup Q) \subseteq P' \cup Q$. An easy inductive argument gives that $kP' \subseteq P' \cup Q$ for all $k \ge 1$; thus $\vs{\cl(P)} = \cup_{k\ge 1}(kP')\subseteq P' \cup Q$ and so $\cl(P) \subseteq P \cup Q$ as required. 
	
Let $A$ be a coset of $\cl(P)$; note that $A \subseteq G \del \cl(P) \subseteq Q \cup R$. If $A$ contains a vector $w \in Q$, then a similar inductive argument gives that $w + kP' \subseteq P' \cup Q$ for all $k \ge 0$ and so $A = \vs{\cl(P)} + w = \cup_{k \ge 0}(w+kP') \subseteq Q$. Otherwise $A \subseteq R$, as required. 

Finally, if $(R_1,R_2)$ is a partition of $R$ as in the hypothesis, then for each coset $A \subseteq R$ of $\cl(P)$, we have $(A \cap R_1) + P' \subseteq A \cap (G \del R_2) = A \cap R_1$. If $A$ contains some $u \in R_1$, then induction gives $u + kP' \subseteq R_1$ for all $k \ge 0$; it follows that $A = \cup_{k \ge 0} (u + kP') \subseteq R_1$. So each coset of $\cl(P)$ that is contained in $R$ is contained in either $R_1$ or $R_2$, as required.
\end{proof}

\subsection*{Targets} 
Recall that $M = (E,G)$ is a target if there are distinct flats $F_0 \subset \dotsc \subset F_k$ of $G$ for which $E$ is the union of $F_{i+1} \del F_{i}$ over all even $i < k$. We show that targets are closed under some basic properties. 

\begin{lemma}\label{targetsnice}
	The class of targets is closed under taking induced restrictions, complementations, and lift-joins. 
\end{lemma}
\begin{proof}
	Note that if $F_0 \subseteq \dotsc \subseteq F_k$ are flats of $G$, not necessarily distinct, for which $E$ is the union of $F_{i+1} \del F_{i}$ over all even $i < k$, then we can remove consecutive pairs of equal flats from the sequence to obtain such a sequence where the flats are distinct, certifying that $M = (E,G)$ is a target. Thus we can allow the $F_i$ in the definition to be equal. It follows from this fact that targets are closed under taking induced restrictions and under complementation (consider the sequence $(\varnothing,F_0, \dotsc, F_k,G)$ if $k$ is odd and $(\varnothing,F_0,\dotsc, F_{k-1},G)$ if $k$ is even).
	
	Finally, suppose that $M_1 = (E_1,G_1)$ and $M_2 = (E_2,G_2)$ are targets; let $F_0 \subseteq \dotsc \subseteq F_s$ be flats of $G_1$ and $K_0 \subseteq \dotsc \subseteq K_t$ be flats of $G_2$ certifying this. By possibly appending a copy of $F_s$ to the end of the first sequence, we may assume that $s$ is odd. Let $K_i' = G_1 \oplus K_i$ for each $i \in \{1, \dotsc, t\}$, so $G_1 \subseteq K_0'$ and $K_0',\dotsc,K_t'$ is a nested sequence of flats of $G_1 \oplus G_2$. If $M = (E,G_1 \oplus G_2) = M_1 \ls M_2$, then 
	\begin{align*}
		E &= E_1 \cup (\vs{G_1} + E_2)\\
		& = \cup_{i}(F_{i+1} \del F_i) \cup \left(\vs{G_1} + (\cup_j(K_{j+1} \del K_j))\right)\\
		&= \cup_{i}(F_{i+1} \del F_i) \cup \left(\cup_j(K_{j+1}' \del K_j')\right),
	\end{align*}
	where the unions are taken over all even $i$ and $j$ with $i < s$ and $j < t$. Now, since $s$ is odd while $F_s \subseteq G_1 \subseteq K_1'$, the sequence $(F_0,F_1,\dotsc,F_s,K_0',\dotsc,K_t')$ certifies that $M$ is a target, as required.
\end{proof}

\section{Large Decomposers}

Before proceeding, we restate Theorem~\ref{structure} in terms that will be more convenient. The equivalence follows from Corollary~\ref{claw_free_lift_join} and Lemma~\ref{rlj}, as well as the obvious fact that the three basic classes are claw-free. 

\begin{theorem}\label{main}
	If $M = (E,G)$ is a claw-free matroid, then either
	\begin{itemize}
		\item $M$ is even-plane,
		\item $M^c$ is triangle-free,
		\item $M$ is a strict PG-sum, or
		\item $M$ has a decomposer.
	\end{itemize}
\end{theorem}

We will prove Theorem~\ref{main} by induction on the dimension of $M$, and we are thus interested in when a decomposer $F$ of some induced submatroid $M|H$ of $M$ extends to a decomposer of $M = (E,G)$ itself. This section and the next address a few special cases of this that turn out to be all we need: namely, where $H$ is a hyperplane of $G$, and the decomposer $F$ is minimal and has dimension either $1$ or $\dim(H)-1$. 

In what follows, we will frequently be obtaining a contradiction by finding a claw. Abusing terminology slightly, we say that a set $X$ is itself a \emph{claw} in a matroid $M = (E,G)$ if $X$ is a three-element linearly independent set with $X = E \cap \cl(X)$. Having such an $X$ is equivalent to having an induced $I_3$-restriction. To check that $\{x,y,z\}$ is a claw in this sense, it suffices to verify that $x,y,z \in E$ and $x+y,y+z,x+z,x+y+z \notin E \cup \{0\}$. To keep our proofs somewhat concise, we will usually just assert that various three-element sets are claws without writing the required checks explicitly; we have endeavoured to ensure that when we do this, the check is easy to perform with information recently established. 

In many lemmas to come, we will consider a partition $(X_0,X_1)$ of a hyperplane $H$ defined by $X_1 = (a+E) \cap H$ and $X_0 = H \del X_1$ for some element $a$ of $G \del H$. This partition can be defined alternatively by its property that for all $x \in H$, we have $x \in X_1$ if and only if $x+a \in E$. 

We will first consider the case in which some hyperplane $H$ has a minimal decomposer $F$ which is itself a hyperplane of $H$. This implies that the coset $H \del F$ is either contained in $E$ or disjoint from $E$; the next two lemmas deal with these subcases. 

\begin{lemma}\label{minimal_hyperplane_case_solid}
Let $M = (E,G)$ be a claw-free matroid, let $H$ be a hyperplane of $G$ and let $F$ be a hyperplane of $H$. If $H \del F \subseteq E$ and $M|F$ has no decomposer, then either
\begin{itemize}
	\item $M^c$ is triangle-free, or 
	\item $M$ has a decomposer $F'$ containing $F$. 
\end{itemize}
\end{lemma}

\begin{proof}
	Let $A = H \del F$. Suppose for a contradiction that $M^c$ has a triangle, and that no flat $F'$ containing $F$ is a decomposer of $M$. 
	Thus $F$ does not decompose $M$, so $F$ has a mixed coset $B$. Fix an element $a \in B \del E$, and let $X_1 = (a+E) \cap H$ and $X_0 = H \del X_1$. Note that $X_1 \cap F \neq \varnothing$ since $B$ is mixed. 
	
	We may assume that $A \cap X_0 \neq \varnothing$, as otherwise $A + a \subseteq E$, which implies that the coset $G \del (F\cup B) = A \cup (A + a)$ of $F \cup B$ is contained in $E$, and so $F \cup B$ decomposes $M$. Let $v \in A \cap X_0$, so $a+v \notin E$. 
	
	\begin{claim}Let $T \subseteq H$ be a triangle. Then 
	\begin{itemize}
		\item if $T \subseteq F$ and $|T \cap X_0| = 2$ then $T \subseteq E$, and 
		\item $T$ does not intersect all three of $X_0\del E,X_1\del E$ and $X_1 \cap E$. 
	\end{itemize}
	
	
	\end{claim}
	
	\begin{subproof}[Subproof]
	For the first part, let $T = \{v_1,v_2,v_3\}$ with $T \cap X_0 = \{v_1,v_2\}$. Note that $\{v\} \cup (T + v)\subseteq A \subseteq E$. If $v_1 \notin E$, then $\{v+v_1,a+v+v_1,v\}$ is a claw if $v+v_1 \in X_1$ and $\{v+v_2,v+v_3,a+v_3\}$ is a claw if $v+v_1 \in X_0$; either case gives a contradiction, so $v_1 \in E$. Symmetrically we have $v_2 \in E$. If $v_3 \notin E$, then $\{v,v+v_3,a+v_3\}$ is a claw if $v + v_3 \in X_0$, and $\{v+v_1,v_1+v_2,a+v+v_3\}$ is a claw if $v+v_3 \in X_1$. Thus $v_3 \in E$.  
	
	For the second part, let $T = \{v_1, v_2, v_3\} \subseteq H$ be a triangle for which $v_1 \in X_0 \del E$, $v_2 \in X_1 \del E$ and $v_3 \in X_1  \cap E$. 
	Then $\{a+v_1,a+v_3,v_3\}$ is claw; thus, there are no such triangles. 
	\end{subproof}
	
	We may now apply Lemma~\ref{PQR} to conclude the following.
	
	\begin{claim}
	$F \cap X_0 \subseteq E$.
	\end{claim}
	
	\begin{subproof}[Subproof]
	Suppose not, so that $(F \del E) \cap X_0 \neq \varnothing$. Let $(P,Q,R) = ((F\del E) \cap X_0,(F \cap E) \cap X_0,F \cap X_1)$. The choice of $a$ implies that $R \ne \varnothing$. This is a partition of $F$, and by the first part of the previous claim, no triangle $T$ of $F$ satisfies $|T \cap P| \ge 1$ and $|T \cap R| = 1$, as such a triangle contains exactly two elements of $X_0$ but is not contained in $E$. Moreover, if $(R_1,R_2) = ((F\del E) \cap X_1,(F \cap E) \cap X_1)$, then no triangle of $F$ intersects $P,R_1$ and $R_2$ by the second part of the claim.
	
	 Thus, we can apply Lemma~\ref{PQR} to the flat $F$ with the given partition, so the flat $F' = \cl(P)$ satisfies $F' \subseteq P \cup Q \subseteq X_0$, and every coset of $F'$ in $F$ is contained in either $Q,R_1$ or $R_2$; none of these sets is mixed with respect to $M$, so it follows that either $F' = \varnothing$, $F' = F$, or $F'$ is a decomposer of $M|F$. The last case contradicts hypothesis. The fact that $R$ is nonempty and $F' \subseteq P \cup Q = F \del R$ implies that $F' \ne F$, and so $F' = \varnothing$. This yields $P = \varnothing$, giving the claim. 
	\end{subproof}
	
	This claim, together with the fact that $H \del F \subseteq E$, implies that each triangle of $G$ containing $a$ must intersect $E$. Recall that by assumption, $M^c$ has a triangle $T_0$. As just observed, we have $a \notin T_0$. For each $x \in T_0$ we therefore have $x + a \in E$, as otherwise $\{a,x+a,x\}$ does not intersect $E$. It follows that $T_0 + a \subseteq E$, so $E \cap \cl(T_0 \cup \{a\}) = T_0 + a$ and $T_0 + a$ is thus a claw, giving a contradiction.
	\end{proof}

Next, we consider the case in which the coset of $H \del F$ is disjoint from $E$. Although the statement is not self-complementary, the first part of its proof is surprisingly similar to the proof of the previous case. 

\begin{lemma}\label{minimal_hyperplane_case_empty}
	Let $M = (E,G)$ be a claw-free matroid, let $H$ be a hyperplane of $G$ and let $F$ be a hyperplane of $H$ with $|F| >1$. If $(H \del F) \cap E = \varnothing$ and $M|F$ has no decomposer, then either 
	\begin{itemize}
		\item $M$ and $M|F$ are both strict PG-sums, or
		\item $M$ has a decomposer $F'$ containing $F$. 
	\end{itemize}
\end{lemma}

	\begin{proof}
	Let $A = H \del F$, so $A \cap E$ is empty. Suppose that neither outcome holds. Thus, $F$ is not a decomposer, so has a mixed coset $B$. Fix an element $a \in B \cap E$, and let $X_1 = (a+E) \cap H$ and $X_0 = H \del X_1$. That $B$ is mixed implies that $X_0 \cap F \neq \varnothing$. If $A \subseteq X_0$, then the set $A \cup (A + \{a\}) = G \del (F \cup B)$ is disjoint from $E$ and so $F \cup B$ is a decomposer of $M$; thus $A \cap X_1 \ne \varnothing$. Let $v \in A \cap X_1$, so $a+v \in E$. We claim the following.
	
	\begin{claim}\label{YYX}
	Let $T \subseteq H$ be a triangle. Then 
	\begin{enumerate}[(i)]
		\item\label{yyxi} if $T \subseteq F$ and $|T \cap X_1| = 2$, then $T \subseteq E$, and 
		\item\label{yyxii} $T$ does not intersect all three of $X_1\del E,X_0\del E$ and $X_0 \cap E$. 
	\end{enumerate}
	\end{claim}
	
	\begin{subproof}[Subproof]
	For (\ref{yyxi}), suppose that $T \subseteq F$ and $|T \cap X_1| = 2$; let $\{v_1,v_2\} = T \cap X_1$ and $\{v_3\} = T \cap X_0$. Note that $v + \vs{T}$ is disjoint from $E$. Now
	\[\cl(\{a,v,v_3\}) \cap E = \{a,a+v\} \cup (E \cap \{v_3,a+v+v_3\}),\]
	and
	\[
	\cl(\{v_3,a+v,a+v_1\}) \cap E = \{a+v_1,a+v_2,a+v\} \cup (E \cap \{v_3,a+v+v_3\}),
	\]
	which implies that $\{v_3,a+v+v_3\} \subseteq E$, as otherwise one of these planes gives a claw in $M$. We have 
	\[
	\cl(\{a+v,v_1,a+v_2\}) \cap E = \{a+v,a+v_2\} \cup (\{v_1,a+v+v_1\} \cap E)
	\]
	and 
	\[
	\cl(\{a,v,v_1\}) \cap E = \{a,a+v_1,a+v\} \cup (\{v_1,a+v+v_1\} \cap E),
	\]
	so, similarly, $\{v_1,a+v+v_1\} \subseteq E$. Thus $v_1 \in E$ and, symmetrically, $v_2 \in E$; therefore $T \subseteq E$. 
	
	For the second part, note that if $T = \{v_1,v_2,v_3\}$ is a triangle with $v_1 \in X_1\del E$. $v_2 \in X_0\del E$ and $v_3 \in X_0 \cap E$, then $\cl(T \cup \{a\}) \cap E = \{a,a+v_1,v_3\}$ and so $M$ has a claw. Thus, there are no such triangles. 
	\end{subproof}
	
	We may now apply Lemma~\ref{PQR} to conclude the following.
	
	\begin{claim}\label{no_empty_X_1}
	$F \cap X_1 \subseteq E$.
	\end{claim}
	
	\begin{subproof}[Subproof]
	Suppose not, so that $(F \del E) \cap X_1 \neq \varnothing$. Define a partition of $F$ by $(P,Q,R) = ((F\del E) \cap X_1,F \cap E \cap X_1,F \cap X_0)$ and let $(R_1,R_2) = ((F\del E) \cap X_0,F\cap E \cap X_0)$. Recall that $R \ne \varnothing$ by the choice of $a$. 
	
	The previous claim gives that no triangle $T$ of $F$ satisfies $|T \cap R| = 1$ and $|T \cap P| \ge 1$, and that no triangle of $F$ intersects all three of $P,R_1,R_2$; thus, Lemma~\ref{PQR} implies that the flat $F' = \cl(P)$ satisfies $F' \subseteq P \cup Q$, while each coset of $F$ is contained in $Q,R_1$ or $R_2$. As before, this implies that $F'$ has no mixed cosets in $F$, so either $F' = \varnothing,F' = F$, or $F'$ is a decomposer of $M|F$; the last case contradicts the hypothesis, and we have $F' \ne F$ because $R \ne \varnothing$; it follows that $F' = \varnothing$ and so $P = \varnothing$, which gives the claim.
	\end{subproof}
	
	We now diverge from the techniques in Lemma~\ref{minimal_hyperplane_case_solid}. 
	The previous claim gives that the sets $X_0 \cap E, X_0 \del E$ and $X_1 \cap E$ induce a partition of $F$. We now show that this partition naturally gives rise to a partition of $A = v + \vs{F}$. 

	\begin{claim}\label{intermediate}For all $u \in F$, 
	\begin{itemize}
		\item If $u \in X_0 \cap E$, then $v+u \in X_1$.
		\item If $u \in X_0 \del E$, then $v+u \in X_0$.
		\item If $u \in X_1 \cap E$, then $v+u \in X_0$.
	\end{itemize}
	\end{claim}
	
	\begin{subproof}[Subproof]
	The first two are immediate; if $u \in X_0 \cap E$ and $v+u \in X_0$, then $\{a,u,a+v\}$ is claw. Moreover, if $u \in X_0 \del E$ and $v+u \in X_1$, then $\{a,a+u+v,a+v\}$ is a claw. 
	
	Finally, suppose for a contradiction that $u \in X_1 \cap E$ but $v+u \in X_1$. Then we claim that $\{u\}$ is a decomposer of $M|F$; since $\{u\} \ne F$ and all cosets of $\{u\}$ are pairs of the form $\{w,w+u\}$, it suffices to show that there does not exist $w \in F \del \{u\}$ for which $w \notin E$ and $u+w \in E$. Consider such a $w$. By ~\ref{no_empty_X_1}, we have $w \in X_0$. If $u+w \in X_1$, then this is a contradiction to \ref{YYX}(\ref{yyxi}), so $u+w \in X_0$. The first two statements of the current claim imply that $v+u+w \in X_1$, and $v+w \in X_0$, but then $\{a,a+v+u+w,a+v+u\}$ is a claw. Thus, $\{u\}$ is a decomposer of $M|F$, contradicting the hypothesis.
	\end{subproof}
	
	\begin{claim}\label{twoflats}
	$F \cap X_1$ and $(F \cap E) \cap X_0$ are flats.
	\end{claim}
	
	\begin{subproof}[Subproof]
	Suppose there is a triangle $T$ of $F$ such that $|T \cap X_1| = 2$. Let $\{v_1,v_2\} = T \cap X_1$ and $\{v_3\} = T \cap X_0$. Then \ref{YYX} gives $T \subseteq E$, and \ref{intermediate} gives $\{v+v_1, v+v_2\} \subseteq X_0$ and $v+v_3 \in X_1$. But this implies that $\{a+v,a+v_1,v_2\}$ is a claw. So there are no such triangles. This implies that $F \cap X_1$ is a flat in $F$. 
	
	Now, suppose that $F$ has a triangle $T$ for which $|T \cap (E \cap X_0)| = 2$. Let $\{v_1,v_2\} = T \cap E \cap X_0$ and $\{v_3\} = T \del \{v_1,v_2\}$. If $v_3 \in E \del  X_0$ then $\{a,v_1,v_2\}$ is a claw. If $v_3 \in X_1$, then $v_3 \in E$, and \ref{intermediate} gives $\{v+v_1,v+v_2\} \subseteq X_1$ and $v + v_3 \in X_0$. Now $\{a,v_1,a+v+v_2\}$ is a claw. It follows that there are no such triangles, so $F \cap E \cap X_0$ is a flat. 	
	\end{subproof}
	
	Let $F_1 = F \cap X_1$ and $F_2 = F \cap E \cap X_2$; we know from the above and ~\ref{no_empty_X_1} that $F_1,F_2$ are flats with $F_1 \subseteq E$. Since $E \cap A = \varnothing$, we have $E \cap H = E \cap F = F_1 \cup F_2$. Moreover, 
	\begin{align*}
	E \del H &= \{a\} \cup (a + X_1) \\ 
	&= \{a\} \cup (a + (X_1 \cap F)) \cup (a + (\{v\} \cup (v + X_0 \cap F \cap E)))\\
	&= (a + \vs{F_1}) \cup ((a + v) + \vs{F_2}),
	\end{align*}
	where the second line uses \ref{intermediate}. It follows that 
	\[E = (E \cap H) \cup (E \del H) =  F_1 \cup F_2 \cup (a + \vs{F_1}) \cup ((a + v) + \vs{F_2}),\] which is the  union of the disjoint flats $\cl(F_1 \cup \{a\})$ and $\cl(F_2 \cup \{a+v\})$, neither of which is contained in $F$. It follows that $M$ and $M|F$ are PG-sums. 
	
	By hypothesis, $\dim(F) > 1$. If $M|F$ is not a strict PG-sum then either $F \subseteq E$, or $E \cap F$ is contained in a hyperplane of $F$. In either case, some hyperplane decomposes $M|F$, a contradiction. So $M|F$ is a strict PG-sum; i.e. $\cl(E \cap F) = F$. But the flat $\{a,v,a+v\}$ is disjoint from $F$, so $\cl(E)$ contains $\cl(F \cup \{a,a+v\}) = G$, and thus $\cl(E) = G$. Since $F \not\subseteq E$ we have $E \ne G$ and so $M$ is therefore also a strict PG-sum. 
\end{proof}

\section{Small Decomposers}
We now handle the cases where $M$ has a hyperplane $H$ for which $M|H$ has a one-element decomposer. These are harder. We first consider the case where this decomposer is contained in $E$. 

\begin{lemma}\label{point_decomposer_solid}
	Let $M = (E,G)$ be a claw-free matroid. If $H$ is a hyperplane of $G$, and $\{b\} \subseteq E$ is a decomposer of $M|H$, then either
	\begin{itemize}
		\item $M^c$ is triangle-free, or
		\item $M$ has a decomposer containing $b$.
	\end{itemize}
\end{lemma}
\begin{proof}
	Let $E^c = G\del E$. Suppose that $M$ has no decomposer; thus, $\{b\}$ has a mixed coset $B$ in $M$; let $a \in B \cap E^c$. Note that $a \notin H$ since $B \not\subseteq H$; let $X_1 = (a+E) \cap H$ and $X_0 = H \del X_1$. By construction we have $b \in X_1$. Fix a hyperplane $F$ of $H$ for which $b \notin F$. Let $F_j = \{v \in F \colon |\{a+v, a+b+v\} \cap E| \equiv j \pmod{2}\}$ for each $j \in \{0,1\}$.
	
	We may assume that $F_1 \subseteq E$; indeed, if there exists $v \in F_1 \cap E^c$, then $\{b, a+b, a+v\}$ gives a claw provided $v \in X_1$, and $\{b, a+b, a+b+v\}$ gives a claw if $v \in X_0$. We now argue that we can apply Lemma~\ref{PQR} to a certain partition of $F$; the proof is unilluminating. 
	
	\begin{claim}
		The sets $Q = X_1 \cap F_0 \cap E$, $R = X_0 \cap F_0 \cap E^c$ and $P = F \del (Q \cup R)$ satisfy the hypotheses of Lemma~\ref{PQR} in the flat $F$. 
	\end{claim}
	
	\begin{subproof}[Subproof]
		Since $F_1 \subseteq E$, the four sets \[(P_1,P_2,P_3,P_4)= (X_0 \cap F_1 \cap E,X_0 \cap F_0 \cap E, X_1 \cap F_1 \cap E, X_1 \cap F_0 \cap E^c)\] form a partition of $P$. Suppose that the claim fails; then $F$ has a triangle $\{v_1, v_2, v_3\}$ for which $v_1 \in P$, $v_2 \in P \cup Q$ and $v_3 \in R$. The fact that $v_3 \in R$ implies that $v_3,v_3+a,v_3+b,v_3+a+b \notin E$. 
		
		Suppose first that $v_2 \in Q$, which gives $v_2,b+v_2,a+v_2,a+b+v_2 \in E$. If $v_1 \in P_1 \cup P_2 = X_0 \cap E$, then $\{v_1, v_2, a+v_2\}$ is a claw and if $v_1 \in F_1 \cap X_1 \cap E$, then $\{b+v_1,b+v_2,a+b+v_2\}$ is a claw. If $v_1 \in F_1 \cap X_0 \cap E^c$, then $\{a+b+v_1,v_2,a+v_2\}$ is a claw. Thus $v_2 \in P$. 
		
		If $v_1,v_2 \in P$ then there exist $i,j$ so that $v_1 \in P_i$ and $v_2 \in P_j$; we may assume by symmetry that $i \le j$. A careful check shows that
		\begin{itemize}
			\item $\{b,a+b+v_1,a+b+v_2\}$ is a claw if $(i,j) = (1,1)$,
			\item $\{v_2,b+v_1,a+b\}$ is a claw if $i \in \{1,2\}$ and $j = 2$,
			\item $\{v_1,v_2,a+v_2\}$ is a claw if $i \in \{1,2\}$ and $j=3$,
			\item $\{b,a+v_1,a+v_2\}$ is a claw if $(i,j) = (3,3)$, 
			\item $\{a+b,a+v_2,b+v_1\}$ is a claw if $(i,j) = (1,4)$,
			\item $\{a+b,v_1,a+v_2\}$ is a claw if $j = 4$ and $i \in \{2,3\}$,
			\item $\{a+b,a+b+v_1,a+b+v_2\}$ is a claw if $(i,j) = (4,4)$,
		\end{itemize} 
		giving a contradiction in all cases. 
	\end{subproof}
	It follows from Lemma~\ref{PQR} that every coset of $\cl(P)$ is contained in $Q$ or in $R$, and that $\cl(P) \subseteq P \cup Q$. 
	
	Suppose that $R \ne \varnothing$, so $\cl(P) \ne F$. We argue that the flat $D = \cl(\{a,b\} \cup P)$ is a decomposer in $M$; let $A$ be a coset of $D$ in $G$. Now $A$ has the form $A_0 + \{0,a,b,a+b\}$ for some coset $A_0$ of $\cl(P)$ in $F$; we either have $A_0 \subseteq Q$ or $A_0 \subseteq R$. If $A_0 \subseteq Q$ then $A = A_0 + \{0,a,b,a+b\} \subseteq Q + \{0,a,b,a+b\} \subseteq E$ by definition of $Q$. If $A_0 \subseteq R$ then $A = A_0 + \{0,a,b,a+b\} \subseteq R + \{0,a,b,a+b\} \subseteq E^c$ by the definition of $R$. In either case, $A$ is not mixed in $M$. Furthermore, since $\cl(P) \ne F$ we have $\dim(\cl(P)) < \dim(G)-2$ and so $\dim(D) \le \dim(\cl(P)) + 2 < \dim(G)$. Thus $D$ is a decomposer of $G$ containing $b$, giving a contradiction.
	
	So $R = \varnothing$; we now argue that $M^c$ is triangle-free. Indeed, the fact that $F = P \cup Q$ implies that for each $v \in F$, each triangle in the plane $\cl(\{v,a,b\})$ contains an element of $E$. Every triangle of $G$ containing $a$ is contained in such a plane, so each triangle of $G$ containing $a$ contains an element of $E$. Thus if $T \subseteq E^c$ is a triangle, then $a \notin T$, and for each $x \in T$ we must have $x + a \in E$. This implies that $T + \{a\}$ is a claw in $M$. So $M^c$ is triangle-free. 
\end{proof}

We now handle the complementary case where $M|H$ has a one-element decomposer that is not contained in $E$. This case turns out to be extremely intricate, and very different from the previous case, even though the matroids involved ostensibly only differ in a single element. We first need a lemma that recognises even-plane matroids. 

\begin{lemma}\label{one_point_lemma}
	Let $M = (E,G)$ be a claw-free matroid, let $H$ be a hyperplane of $G$, and let $b \in H\del E$ be such that, for all $u \in G \del \{b\}$, we have $|E \cap \{u,u+b\}| = 1$ if and only if $u \notin H$. Then, for every hyperplane $F$ of $H$ not containing $b$, either
	\begin{itemize}
		\item $M|F$ is even-plane, or
		\item $M|F$ is a Bose-Burton geometry. 
	\end{itemize}
\end{lemma}
\begin{proof}
	
	Let $F$ be a hyperplane of $H$ not containing $b$, and suppose that $M|F$ is not even-plane. Let $a \in G \del H$; we have $|E \cap \{a,a+b\}| = 1$ by hypothesis; we may assume that $a \in E$ and $a+b \notin E$. 
	
	Let $X_1 = (a+H) \cap X_1$ and $X_0 = H \del X_1$.  The hypotheses imply that if $v \in H \del \{b\}$, then $\{v,b+v\} \cap E$ has even size, and $\{v,b+v\} \cap X_0$ has odd size.  The next claim essentially states that given a flat $K$ of $H$, there is another flat $K'$ where $M|K = M|K'$, and the elements of $K'$ have some desired intersction with $X_0$ and $X_1$.
	
	\begin{claim}\label{twiddle}
		If $K$ is a flat of $H$ with $b \notin K$, and $I_0,I_1$ are disjoint subsets of $K$ for which $I_0\cup I_1$ is linearly independent, then there is a flat $K'$ of $H$, not containing $b$, and an isomorphism $\psi$ from $M|K$ to $M|K'$ for which $\psi(I_0) \subseteq X_0$ and $\psi(I_1) \subseteq X_1$. 
	\end{claim}
	\begin{subproof}[Subproof]
		By replacing $I_0,I_1$ with supersets if necessary, we may assume that $I_0 \cup I_1$ is a basis for $K$. For each $i \in \{0,1\}$ and $x \in I_i$, let $\psi(x)$ be the unique element of $\{x,x+b\} \cap X_i$, and extend $\psi$ linearly to all $x \in K$. Let $K' = \psi(K)$. The linear independence of $I_0 \cup I_1 \cup \{b\}$ implies that $\psi$ is injective. Moreover, it is clear that $\psi(v) \in \{v,v+b\}$ for all $v \in K$ and so, since $|E \cap \{v,v+b\}|$ is even, $\psi$ is an isomorphism from $M|K$ to $M|K'$ that has the required property by construction.
	\end{subproof}

	\begin{claim}\label{evenplanes}
		If $P$ is a plane of $H$ not containing $b$, then either $P \subseteq E$, or $|P \cap E|$ is even. 
	\end{claim}
	
	\begin{subproof}[Subproof]
		Let $P$ be a counterexample. Suppose first that $|P \cap E| = 5$, so $P \cap E = \{v_1,v_2,v_3,v_1+v_2,v_1+v_3\}$ for some linearly independent $v_1,v_2,v_3$; by \ref{twiddle} with $(I_0,I_1) = (\{v_1,v_2,v_3\},\varnothing)$, we may assume that $v_1,v_2,v_3 \subseteq X_0$. Now 
		\[
		E \cap \cl(\{a,v_2,v_3\}) = \{a,v_2,v_3\} \cup (E \cap \{a+v_2+v_3\}),
		\]
		which implies that $a+v_2+v_3 \in E$, and 
		\[
		E \cap \cl(\{a,v_1,a+v_2+v_3\}) = \{a,v_1,a+v_2+v_3\} \cup (E \cap \{a+v_1+v_2+v_3\}),
		\]
		which gives $a+v_1+v_2+v_3 \in E$. But now $v_1+v_2,v_1+v_3$, and $a+v_1+v_2+v_3$ give a claw in $M$.
		
		Suppose now that $|P \cap E| = 3$. Since $M|P$ is not a claw, we must have $E \cap P = \{v_1,v_2,v_1+v_2\}$ for some $v_1,v_2$; let $v_3 \in P\del E$. By \ref{twiddle} we may assume that $v_1,v_3 \in X_0$ and $v_2 \in X_1$. This implies that $b+v_2 \in X_0$ and $b+v_1,b+v_3 \in X_1$. Now 
		\[
		E \cap \cl(\{a,v_1,v_3\}) = \{a,v_1\} \cup (E \cap \{a+v_1+v_3\}), 
		\]
		which implies that $a+v_1+v_3 \notin E$, and 
		\[
		E \cap \cl(\{a,b+v_2,v_3\}) = \{a,b+v_2\} \cap (E \cap \{a+b+v_2+v_3\})
		\]
		giving $a+b+v_2+v_3 \notin E$; thus $a+v_2+v_3 \in E$. Now the set $\{v_1+v_2,a+v_2,a+v_2+v_3\}$ is a claw.
		
		Finally, suppose that $|P \cap E| = 1$; by \ref{twiddle} we may assume that $P = \cl(\{v_1,v_2,v_3\})$ where $\{v_1\} = P \cap E$ and $v_1 \in X_0$ while $v_2,v_3 \in X_1$. Then $\{v_1+v_2,v_1+v_3,v_2+v_3\} \subseteq X_1$, as otherwise one of $\{a,v_1,a+v_2\}$, $\{a,v_1,a+v_3\}$ or $\{a,a+v_2,a+v_3\}$ is  a claw. Hence $v_1+v_2+v_3 \in X_1$, as otherwise $\{a,v_1,a+v_2+v_3\}$ is a claw. But now the triple $\{a+v_1+v_2,a+v_1+v_3,a+v_1+v_2+v_3\}$ is a claw, completing the contradiction. 
	\end{subproof}
	
	By the above claim and the assumption that $M|F$ is not even-plane, we can conclude that $E$ contains a plane of $F$. Let $K$ be a largest flat of $F$ for which $K \subseteq E$; by the above, $\dim(K) \ge 3$. If $K = F$ then $M|F$ is a Bose-Burton geometry, as required. Otherwise, let $v \in F \del E$. Let $K_1 = \{(v+E) \cap K\}$ and $K_0 = K \del K_1$. If some triangle $T$ of $K$ has even intersection with $K_1$, then the plane $\cl(T \cup \{v\})$ contains an odd number of elements of $E$ and contains the nonelement $v$ of $E$, contradicting \ref{evenplanes}. Thus, every triangle of $K$ has odd intersection with $K_1$; it follows that either $K_1$ is a hyperplane of $K$, or $K_1 = K$. 
	
	If $K_1$ is a hyperplane of $K$, then since $\dim(K) \ge 3$, there is some triangle $T \subseteq K_1$ and some $w \in K_0$. Now $T + w \subseteq K_0$ and so $E \cap \cl(T \cup \{v+w\}) = T$, so $|E \cap \cl(T \cup \{v+w\})| = 3$, contradicting \ref{evenplanes}. So $K_1 = K$, and thus $K + v \subseteq E$. This argument applies for every $v \in F\del E$, and by the maximality of $K$, every coset of $K$ contains such a $v$. Therefore every coset $A$ of $K$ satisfies $|A \cap E| = |A|-1 = |K|$. It follows that $|E \cap F| = 2^{\dim(F)}(1-2^{-\dim(K)})$. Since $E \cap F$ contains no flat of dimension larger than $\dim(K)$, Theorem~\ref{bbt} implies that $M|F$ is a Bose-Burton geometry, as required. 
\end{proof}

We can now deal with the case where $M|H$ has a one-element decomposer disjoint from $M$. 

\begin{lemma}\label{point_decomposer_empty}
	Let $M = (E,G)$ be a claw-free matroid and let $H$ be a hyperplane of $G$. If $\{b\} \subseteq H\del E$ is a decomposer of $M|H$, then either 
	\begin{itemize} 
		\item $M^c$ is triangle-free, 
		\item $M$ is even-plane, or
		\item $M$ has a decomposer.
	\end{itemize}
\end{lemma}

\begin{proof}
	Suppose that $M^c$ has a triangle and $M$ has no decomposer, so $\{b\}$ has a coset $\{a,a+b\}$ in $M$ where $\{a,a+b\} \cap E = \{a\}$. We argue throughout a series of claims that $M$ is even-plane. Since $\{b\}$ is a decomposer of $H$, we have $a \notin H$; as usual, let $X_1 = (a+H) \cap E$ and $X_0 = H \del X_1$. Note that $b \in X_0$. 
	Let $E^c = G\del E$ and define a partition $(H_0,H_1)$ of $H \del \{b\}$ by \[H_i = \{v \in H \del \{b\} \colon |\{a+v,a+b+v\} \cap E| \equiv i \pmod{2}\}.\]

	If $H  \del \{b\} \subseteq E$ then recall that $E^c$ contains a triangle $T$ of $G$. This $T$ intersects $H$ and so $b \in T$, which implies that $a \notin T$, giving $\cl(T \cup \{a\}) \cap E = \{a\}+T$, which yields a claw in $M$. This is a contradiction, so $\{b\}$ has a coset that is disjoint from $E$.

	If $v \in H_0 \cap X_0 \cap E$ then $\{a,v,b+v\}$ is a claw. If $v \in H_0 \cap X_1 \cap E^c$ then $\{a,a+b+v,a+v\}$ is a claw; it follows that $H_0 \subseteq E \Delta X_0$, and so for every $v \in H_0$ the set $v + \{0,a,b,a+b\}$ is unmixed in $M$. 
	
	
	\begin{claim}
		$H_1 \not\subseteq E$. 
	\end{claim}
	\begin{subproof}
		Suppose that $H_1 \subseteq E$. Then the sets \[(P,Q,R) = (H_1 \cap E,H_0 \cap X_1 \cap E,H_0 \cap X_0 \cap E^c)\] partition $H$. Let $F$ be a hyperplane of $H$ not containing $b$, and let $(P',Q',R') = (P \cap F,Q \cap F,R \cap F)$; we just saw that $H\del E$ contains a coset of $\{b\}$ which implies that $R$ contains such a coset, so $R' \ne \varnothing$. We show that the partition $(P',Q',R')$ of $F$ satisfies the hypotheses of Lemma~\ref{PQR}. If it does not, then $F$ has a triangle $\{v_1,v_2,v_3\}$ with $v_1 \in P,v_2 \in P \cup Q$ and $v_3 \in R$. If $v_2 \in Q$ and $v_1 \in X_0$, then $\{a+v_2,a+b+v_2,a+b+v_1\}$ is a claw. If $v_2 \in Q$ and $v_1 \in X_1$, then $\{a+v_2,a+b+v_2,a+v_1\}$ is a claw. If $v_2 \in P$ then $\{a,i_1b + v_1,i_2b + v_2\}$ is a claw, where $i_1$ and $i_2$ are the binary scalars for which $v_1 \in X_{i_1}$ and $v_2 \in X_{i_2}$. So Lemma~\ref{PQR} applies to $(P',Q',R')$ in $F$, giving $\cl(P') \subseteq P' \cup Q'$, while every coset of $\cl(P')$ in $F$ is contained in either $Q'$ or $R'$. 
		
		We now argue that $D = \cl(\{a,b\} \cup P')$ is a decomposer of $M$. Clearly $D \ne \varnothing$, and since $\cl(P') \subseteq (P' \cup Q')$ and $R' \ne \varnothing$, we have $\cl(P') \ne F$ which implies that $\dim(\cl(P')) < \dim(F)$ and $\dim(D) < \dim(M)$. Consider a coset $A$ of $D$; now $A = A_0 + \{0,a,b,a+b\}$ for some coset $A_0$ of $\cl(P')$ in $F$. If $A_0 \subseteq Q'$ then $A = A_0 + \{0,a,b,a+b\} \subseteq Q + \{0,a,b,a+b\} \subseteq E$ by the definition of $Q$, and if $A_0 \subseteq R'$ we have $A \subseteq E^c$ by the definition of $R$. Therefore $D$ has no mixed cosets in $G$, and is thus a decomposer of $M$; this is a contradiction. 
	\end{subproof}

	\begin{claim}\label{epbb}
		For each flat $K$ of $H$ with $K \subseteq H_1$, either $M|K \in \cE_3$ or $M|K$ is a Bose-Burton geometry. 
	\end{claim}
	\begin{subproof}
		Let $G' = \cl(K \cup \{a,b\})$ and $H' = G' \cap H$. Let $M' = M|G'$; we apply Lemma~\ref{one_point_lemma} to $b$ and $H'$ in $G'$. Let $v \in G' \del \{b\}$. If $v \in H'$ then $|\{v,b+v\} \cap E| \in \{0,2\}$ since no coset of $\{b\}$ in $H$ is mixed. If $v \in \{a,a+b\}$ then $|\{v,b+v\} \cap E| = 1$ and, if $v \in (G' \del H')  \del \{a,a+b\}$, then the plane $\cl(\{a,b,v\})$ intersects $K$ in an element $w$ for which $v \in \{w+a,w+a+b\}$; since $w \in H_1$ this implies that $|\{v,b+v\} \cap E| = 1$. Thus, for all $v \in G' \del \{b\}$, we have $|\{v,b+v\} \cap E| = 1$ if and only if $v \notin H$, and the claim follows from Lemma~\ref{one_point_lemma}. 
	\end{subproof}

	Define a partition $(P,Q,R_1,R_2)$ of $H \del \{b\}$ by  
	\[(P,Q,R_1,R_2) = (H_1 \cap E^c,H_1 \cap E,H_0 \cap X_0 \cap E^c,H_0 \cap X_1 \cap E).\]
	Note that since $H_1 \not\subseteq E$ we have $P \ne \varnothing$.

	\begin{claim}\label{PQRclaim}
		$G$ has no triangle $T$ with $|T \cap P| \ge 1$ and $|T \cap (R_1 \cup R_2)| = 1$, and there is no triangle of $G$ that intersects $P,R_1$ and $R_2$. 
	\end{claim}
	\begin{subproof}
		Let $R = R_1 \cup R_2$. The definition of the $H_i$ and the fact that $\{b\}$ is a decomposer imply that each set $W \in \{P,Q,R_1,R_2,R_1 \cup R_2\}$ satisfies $W= \{b\} + W$.
		
		To see the first part, let $\{v_1,v_2,v_3\}$ be a triangle of $G$ with $v_1 \in P$, $v_2 \in P \cup Q$ and $v_3 \in R$. For each $i \in \{1,2\}$, the set $\{v_i,v_i+b\}$ intersects $X_0$ in exactly one element $w_i$, and we have $w_1 + w_2 \in \{v_1+v_2,v_1+v_2 + b\} \subseteq R$, so $\{w_1,w_2,w_3=w_1 + w_2\}$ is a triangle with $w_1 \in P \cap X_0, w_2 \in (P \cup Q) \cap X_0$ and $w_3 \in R$. 
		
		Suppose that $w_2 \in Q$. If $w_3 \in R_1$ then $\{a,a+b+w_1,w_2\}$ is a claw. If $w_3 \in R_2$ then $\{a+w_3,a+b+w_3,a+b+w_2\}$ is a claw. Suppose now that $w_2 \in P$. If $w_3 \in R_1$ then $\{a,a+b+w_1,a+b+w_2\}$ is a claw. If $w_3 \in R_2$ then $\{a+w_3,a+b+w_3,a+b+w_2\}$ is a claw, completing the contradiction. 
		
		For the second part, consider a triangle $\{v_1,v_2,v_3\}$ of $G$ where $v_1 \in P$, $v_2 \in R_1$ and $v_3 \in R_2$. Let $\{w_1\} = \{v_1,v_1+b\} \cap X_0$; let $w_2 = v_2$ and $w_3 = w_1 + v_2$; since $w_i \in \{v_i,v_i+b\}$, the set $\{w_1,w_2,w_3\}$ is a triangle with $w_1 \in P \cap X_0$ while $w_2 \in R_1$ and $w_3 \in R_2$. Now $\{a+b+w_1,a+w_3,b+w_3\}$ is a claw. 
	\end{subproof}
	
	Let $w \in P$ and let $z$ be the element of $\{w,b+w\} \cap X_1$, noting that $z \in P$. Let $F$ be a hyperplane of $H$ containing $z$ but not $b$. Let $(P',Q',R_1',R_2')$ be the partition of $F$ induced by $(P,Q,R_1,R_2)$. By Lemma~\ref{PQR} and \ref{PQRclaim}, we have $\cl(P') \subseteq P' \cup Q' = F_1$, and every coset of $\cl(P')$ in $F$ is contained in $Q',R_1'$ or $R_2'$. Let $F' = \cl(P')$. 

	\begin{claim}\label{mpbb}
		$M|F'$ is not a Bose-Burton geometry. 
	\end{claim}
	\begin{subproof}
		Suppose that $M|F'$ is a Bose-Burton geometry. Let $F'' = P' \cap X_1$. This set is nonempty because $z \in F''$. We will show that $F''$ is a decomposer of $M$; we first argue that it is a flat. Indeed, if $x,y \in F''$ are distinct then, since $x,y \in E^c$, the fact that $M|F'$ is a Bose-Burton geometry gives $x+y \in E^c$, and since $F'$ is a flat we have $x+y \in F'$, so $x+y \in F' \cap E^c$, but this gives $x+y \in X_1$, as otherwise $\{a,a+x,a+y\}$ is a claw. Therefore $x+y \in F' \cap E^c \cap X_1 = F''$, and thus $F''$ is a nonempty flat. 
		
		Let $A$ be a coset of $F''$ in $F$; we now show that $A$ is contained in $E$ or $E^c$, is contained in $X_i$ for some $i$, and is contained in $F_j$ for some $j$. To see this, we consider two cases:
		\begin{itemize}
			\item If $A \not\subseteq F'$ then $A$ is contained in some coset of $F'$ in $F$ and is thus contained in $R_1,R_2$ or $Q$; in the first two cases the conclusion is clear, in the last case we have $A \subseteq E$ and $A \subseteq F_1$, and if $x_0 \in X_0 \cap A$ and $x_1 \in X_1 \cap A$ then $x_0 + x_1 \in F'' \subseteq X_1 \cap E^c$ and it follows that $\{a,x_0,b+x_1\}$ is a claw; thus $A \subseteq X_0$ or $A \subseteq X_1$. 
			\item If $A \subseteq F'$ then $A \subseteq F_1$, and since $M|F'$ is a Bose-Burton geometry while $F'' \subseteq E^c \cap F'$, we clearly have $A \subseteq E$ or $A \subseteq E^c$. If $A \subseteq E^c$ then $A \subseteq P'$; since $A$ is disjoint from $F'' = P' \cap X_1$ we have $A \subseteq X_0$. If $A \subseteq E$ and $A$ intersects $X_0$ in $x_0$ and $X_1$ in $x_1$, then $\{a,x_0,b+x_1\}$ is a claw; thus $A \subseteq X_0$ or $A \subseteq X_1$. 
		\end{itemize}
		Finally, we show that $F''$ is a decomposer in $M$. For this, we need to show that every coset of $F''$ in $M$ is contained in $E$ or in $E^c$; let $B = \vs{F''} + u$ be such a coset, where $u \in G \del F''$. Since $\cl(F \cup \{a,b\}) = G$ we have $u = v + h$ for some $v \in F$ and $h \in \{0, a,b,a+b\}$. Therefore $B = (\vs{F''} + v) + h$. 
		
		If $v \notin F''$ then $\vs{F''} + v$ is a coset $A$ of $F''$ in $F$, so $A$ is contained in either $E$ or $E^c$, is contained in $X_i$ for some $i \in \{0,1\}$, and is contained in $F_i$ for some $i \in \{0,1\}$. The fact that $H \cap E = (H \cap E) + b$, together with the definition of the $X_i$ and $F_i$, thus imply that each of the sets $A,A+a,A+b,A+(a+b)$ is contained in either $E$ or $E^c$. But $B$ is one of these sets, so $B \subseteq E$ or $B \subseteq E^c$. 
		
		If $v \in F''$ then $B = \vs{F''} + h$ for some $h \in \{a,b,a+b\}$. Recall that $F'' \subseteq F_1 \cap X_1 \cap E^c$. If $h = a$ then the fact that $F'' \subseteq X_1$ and $a \in E$ gives $B \subseteq E$. If $h = b$ then the fact that $\{b\}$ is a decomposer of $M|H$ and $b \notin H$ gives $B \subseteq E^c$. If $h = a+b$ then, since $F'' \subseteq F_1 \cap X_1$ and $a+b \notin E$, we have $B \subseteq E^c$. Therefore $F''$ is a decomposer of $M$, giving the needed contradiction.
	\end{subproof} 
	
	Let $A$ be a coset of $\cl(P')$ in $F$, so $A$ is contained in $Q,R_1$ or $R_2$. If $A \subseteq Q$ then $A \subseteq F_1 \cap E$ and so the flat $\cl(P') \cup A$ is contained in $F_1$. By \ref{epbb} it follows that $M|(\cl(P') \cup A)$ is either a Bose-Burton geometry or is in $\cE_3$. In the first case, $M|\cl(P')$ is a Bose-Burton geometry, contradicting \ref{mpbb}. In the second case, \ref{mpbb} gives that there is a triangle $T \subseteq \cl(P')$ for which $|T \cap E| = 1$, but then the fact that $A \subseteq E$ yields $|\cl(T \cup \{u\}) \cap E| = 5$ for every $u \in A$, contradicting $M|(\cl(P') \cup A) \in \cE_3$. Therefore every coset of $\cl(P')$ in $F$ is contained in $R_1$ or $R_2$. 
	
	\begin{claim}
		$F' = F = F_1$. 
	\end{claim}
	\begin{subproof}
		Since $F' \subseteq F_1 \subseteq F$, it suffices to show that $F' = F$; suppose not, so $\dim(F') \le \dim(G) - 3$, and therefore $\cl(F' \cup \{a,b\}) \ne G$. We show that $\cl(F' \cup \{a,b\})$ is a decomposer of $M$. Let $C$ be a coset of $\cl(F' \cup \{a,b\})$, so $C$ has the form $C = \vs{F'} + \{0,a,b,a+b\} + u$ for some $u \notin \cl(F' \cup \{a,b\})$; by replacing $u$ by $u+a$, $u+b$ or $u+a+b$ if necessary, we may assume that $u \in F$, so $C = (\vs{F'} + u) + \{0,a,b,a+b\}$. The coset $\vs{F'} + u$ of $F'$ is contained in $R_1$ or in $R_2$ by the above observations, so $\vs{F'} + u$ is contained in $E$ or $E^c$, and it follows from that fact that $\{b\}$ is a decomposer and the definition of the $X_i$ and $F_i$ that for all $x \in \vs{F'} + u$, we have $x + \{0,a,b,a+b\} \subseteq E$ for all $x \in E$, and $x + \{0,a,b,a+b\} \subseteq E^c$ for all $x \in E^c$. Therefore the coset $C$ is contained in either $E$ or $E^c$. Since $F'$ is a proper nonempty flat of $G$, $\cl(F' \cup \{a,b\})$ is thus a decomposer of $M$, giving a contradiction. 
	\end{subproof}

	
		By \ref{epbb} and \ref{mpbb}, we have $M|F = M|F' \in \cE_3$. Since $M|H$ is the doubling of $M|F$, it follows that $M|H \in \cE_3$ also. The fact that $F = F_1$ implies that $H \del \{b\} = H_1$ and so $X_0 + b = X_1$. Let $J_i = (H  \del \{b\}) \cap (E \Delta X_i)$ for each $i \in \{0,1\}$, so $(J_0,J_1)$ is a partition of $H \del \{b\}$ for which $J_0 + b = J_1$. 
	\begin{claim}
		If $T$ is a triangle of $H$ with $b \notin T$, then $|T \cap J_0|$ is even.
	\end{claim}
	\begin{subproof}[Subproof]
		Let $T = \{v_1,v_2,v_3\}$ be a triangle of $H$ with $b \notin T$ for which $T \cap J_0$ is odd and as large as possible. If $v_2,v_3 \notin J_0$ then $\{v_1,v_2+b,v_3+b\}$ is a triangle contained in $J_0$, contradicting maximality. Thus $T \subseteq J_0$. 
		
		If $v_1 \in X_0\del E$ then $\{a,b + v_2 + (1-i_2)a,b+v_3 + (1-i_3)a\}$ is a claw, where $i_2,i_3$ are the binary scalars for which $v_2 \in X_{i_2}$ and $v_3 \in X_{i_3}$. Therefore $T \subseteq X_1 \cap E$. 
		
		Since $M|H \in \cE_3$, it contains no $F_7$-restriction, and since by \ref{mpbb} it is not a Bose-Burton geometry, Theorem~\ref{bbt} gives that $|E \cap H| < \tfrac{3}{4} \cdot 2^{\dim(H)}$. The triangle $T$ is contained in exactly $2^{\dim(H)-2}-1$ planes of $H$; since $3 + 3(2^{\dim(H)-2}-1) > |E \cap H|$, a majority argument gives that there is a plane $W$ of $H$ containing $T$ for which $|(W \del T) \cap E| < 3$; thus $|W \cap E| = |T| + |(W \del T) \cap E| < 6$; since $M|H \in \cE_3$ this gives $|W \cap E| \le 4$, and so $W \cap E$ contains exactly one element $w$ outside $T$, and $W = \cl(T \cup \{w\})$. Every element of $W \cap E^c$ lies in a triangle of $W$ containing exactly one element of $E$, so $b \notin W$. The three-element set $T+w$ intersects either $X_0$ or $X_1$ in two elements; say they are $v_1 + w$ and $v_2 + w$. If $v_1+w,v_2 +w \in X_0$ then $\{a,a+b+v_2+w,b+v_3\}$ is a claw. If $v_1+w,v_2+w \in X_1$ then $\{a,a+w+v_1,b+v_3\}$ is a claw. 
	\end{subproof} 

	The above claim implies that each triangle of $F$ has even intersection with $J_0$, so $F \cap J_1$ is either equal to $F$, or is a hyperplane of $F$. The definitions of $J_0$ and $J_1$ imply that if $v \in F$, then  $\{v,v+a\} \cap E$ has even size if and only if $v \in F \cap J_0$. Let $M' = (E', \cl(F \cup \{a\}))$ be the doubling of $M|F$ by $a$ if $F \cap J_1 = F$, or  the semidoubling of $M|F$ by $a$ with respect to the hyperplane $F \cap J_1$ if $F \cap J_1$ is a hyperplane of $F$; since $M|F \in \cE_3$, the matroid $M'$ is even-plane by Theorem~\ref{even_plane_doubling_semidoubling}. Since $E \cap \cl(F \cup \{a\}) = E' \Delta (\cl(F \cup \{a\}) \del F)$, by Lemma~\ref{sym_diff_even_plane} it follows that $M | \cl(F \cup \{a\}) \in \cE_3$. 
	
	Moreover, if $x \in F$ then $\{x,b+x\} \cap E$ has even size, and if $x \in \vs{F}+a$ then $\{x,b+x\} \cap E = \{(x+a) + a, (x+a) + a+b \} \cap E$, which has odd size because $x+a \in F = F_1$. Since $b \notin \cl(F \cup \{a\})$ and $F$ is a hyperplane of $\cl(F \cup \{a\})$, it follows that $M$ is the semidoubling of $M|\cl(F \cup \{a\})$ by $b$ with respect to the hyperplane $F$. Thus $M \in \cE_3$, as required. 
\end{proof}

\section{General decomposers}

We now combine the results in the previous two sections to completely describe the claw-free matroids having a hyperplane that admits a decomposer. 

\begin{theorem}\label{hyperplane_decomposer}
	Let $M = (E,G)$ be a claw-free matroid and let $H$ be a hyperplane of $G$ for which $M|H$ has a decomposer. Then either
	\begin{itemize}
		\item $M$ is even-plane,
		\item $M$ is a PG-sum,
		\item $M^c$ is triangle-free, or 
		\item $M$ has a decomposer. 
	\end{itemize}
\end{theorem}
\begin{proof}


Suppose that none of the outcomes hold. Let $F \subseteq H$ be a minimal decomposer of $M|H$. The minimality of $F$ implies 
\begin{claim}\label{Fmin} $M|F$ has no decomposer. \end{claim}
If $F$ is a hyperplane of $H$ or $|F| = 1$, then one of Lemmas ~\ref{minimal_hyperplane_case_solid}, ~\ref{minimal_hyperplane_case_empty},~\ref{point_decomposer_solid} or~\ref{point_decomposer_empty} yields a contradiction.
Therefore $\dim(F) \geq 2$, and $F$ has more than one coset in $H$. Since $F$ does not decompose $M$, it has a mixed coset $B$ in $G$.  

We say that a set $A \subseteq G$ is \emph{vacant} if $A \cap E = \varnothing$, and \emph{full} if $A \subseteq E$. The fact that $F$ decomposes $M|H$ implies that every coset of $F$ in $H$ is either vacant or full. Note that if $A$ and $A'$ are distinct cosets of $F$ in some flat $F^+$ containing $F$, then $A+A'$ is also a coset of $F$ in $F^+$. 

Call a coset $A$ of $F$ in $M \vert H$ \emph{good} if the cosets $A$ and $A+B$ are either both vacant or both full, and say $A$ is $\emph{bad}$ otherwise. 


\begin{claim}
	$F$ has a bad coset in $H$. 
\end{claim}
\begin{subproof}[Subproof]
	Suppose not; we argue that the flat $F \cup B$ decomposes $M$. Indeed, if $A$ is a coset of $F \cup B$ then $A = A' \cup (A' + B)$ for some coset $A'$ of $F$ in $H$; since $A'$ is good this implies that $A$ is vacant or full. Thus $F \cup B$ decomposes $M$, a contradiction. 
\end{subproof}

\begin{claim}\label{badstructure}
	Let $A$ be a bad coset of $F$ in $H$ and let $F_A = F \cup A \cup B \cup (A+B)$. 
	\begin{itemize}
		\item If $A$ is vacant, then $M|F_A$ and $M|F$ are strict PG-sums, and 
		\item if $A$ is full, then $(M|F_A)^c$ is triangle-free. 
	\end{itemize}
\end{claim}
\begin{subproof}
	Note that $F_A$ is a flat of $G$ such that $F \cup A$ is a hyperplane of $F_A$, and $F$ is a hyperplane of $F \cup A$. Let $F^+$ be a flat of $F_A$ that contains $F$. It is easy to see that $F^+$ has a coset containing $B$ or $A \cup (B+A)$. But $B$ is mixed, and the fact that $A$ is bad implies that $A \cup (B+A)$ is mixed. So $F^+$ is not a decomposer of $M|F_A$, and thus no decomposer of $M|F_A$ contains $F$. By \ref{Fmin}, we can apply Lemma~\ref{minimal_hyperplane_case_solid} (if $A$ is full) or Lemma~\ref{minimal_hyperplane_case_empty} (if $A$ is vacant) to obtain the desired conclusion.
\end{subproof}

\begin{claim}\label{bothfull}
	If $A_1,A_2$ are distinct full cosets of $F$ in $H$ and $A_1 + A_2$ is vacant, then $A_1$ and $A_2$ are good. 
\end{claim}
\begin{subproof}
	Suppose otherwise; we may assume that $A_1$ is bad, so there exists $u_1 \in (A_1 + B)\del E$. Let $A_3 = A_1 + A_2$ and $F_i = F \cup A_i \cup B \cup (A_i + B)$ for each $i \in \{1,2,3\}$. By \ref{badstructure}, the matroid $(M|F_1)^c$ is triangle-free. 
	Let $a \in B\del E$ and let $v_1 = a + u_1 \in A_1$. Since $\dim(F) > 1$ and $M|F$ has no decomposer, there exists $x \in F\del E$. Since $(M|F_1)^c$ is triangle-free we have $x+a \in E$ and $x + a+v_1 \in E$. 
	
	Let $v_3 \in A_3$. If both $a+v_3$ and $a+x+v_3$ are both nonelements of $E$, then $\{v_1+v_3,x+v_1+v_3,a+x+v_1\}$ is a claw, and if both are elements of $E$, then $\{a+v_3,a+x+v_3,a+x\}$ is a claw. Thus exactly one is an element of $E$; by possibly replacing $v_3$ by $v_3 + x$, we may assume that $a+v_3 \in E$ and $a+x+v_3 \notin E$.
	If $a+x+v_1+v_3 \in E$, then $\{a+x+v_1+v_3,a+x+v_1,a+v_1+v_3\}$ is a claw, so $a+x+v_1+v_3 \notin E$. 
	
	Since $A_3$ is vacant and $a+v_3 \in E$, the set $A_3$ is bad. By \ref{badstructure} the matroid $M|F_3$ is a PG-sum. Thus $E \cap F_3$ is the disjoint union of two flats $K_1,K_2$. Since $a+x,a+v_3 \in E \cap F_3$ and $(a+x) + (a+v_3) = x+v_3 \in A_3$ is not, one of these two flats (say $K_1$) contains $a+x$, and the other (say $K_2$) contains $a+v_3$. Since $K_1,K_2$ are disjoint flats with union $E \cap F_3$, we have $(K_1 + K_2) \cap E = \varnothing$.
	
	If $K_1 \cap F = \varnothing$, then $F \cap E = F \cap K_2$ which is a flat of $F$. But since $\dim(F) \ge 2$, it follows that either $F \cap E = \varnothing$, in which case every $1$-dimensional flat of $F$ decomposes $M|F$, or $F \cap E$ is a nonempty flat of $F$, in which case every singleton in this flat decomposes $F$. Either case contradicts hypothesis, so it follows that $K_1 \ne \{a+x\}$ and therefore $K_1 \cap F$ contains an element $w$. 
	
	So $w \in E$ and, since $K_1$ is a flat, we have $w+a+x \in E$. Moreover, we have $a+w+v_3 = w + (a+v_3) \in K_1+K_2$, so $a+w+v_3 \notin E$. Now using the fact that $a+x+v_1+v_3 \notin E$ as observed earlier, the set $\{a+x+w,x+w+v_1,w+v_1+v_3\}$ is a claw. 
\end{subproof}

\begin{claim}\label{bothvacant}
	Let $A_1,A_2$ be distinct vacant cosets of $F$ in $H$. Then $A_1$ or $A_2$ is good. Moreover, if $A_1 + A_2$ is full, then $A_1$ and $A_2$ are both good. 
\end{claim}
\begin{subproof}
	Let $a \in B\cap E$. Let $I$ be the set of $i \in \{1,2\}$ for which $A_i$ is bad. It suffices to show that if $|I| \ge 1$, then $|I| = 1$ and $A_1+A_2$ is vacant; let $A_3 = A_1 + A_2$, and suppose that $|I| \ge 1$; we may assume that $1 \in I$. For each $i \in I$, let $F_i = F \cup B \cup A_i \cup (A_i+B)$. By \ref{badstructure} the matroids $M|F_i$ and $M|F$ are strict PG-sums. Let $K_i,L_i$ be the summands of $M|F_i$, where $a \in K_i$; since $M|F$ is strict, the sets $K_i \cap F$ and $L_i \cap F$ are nonempty.
	
	Specialising to $i = 1$, let $x \in (K_1 + L_1) \cap F $; thus, $x \notin E$, and, since $x+ a \in K_1 + L_1$, we have $x+a \notin E$.

		
	Again, consider a general $i \in I$. Since $A_i$ is bad, the set $A_i + B = A_i + \{a\}$ contains an element $v_i+a$ of $E$, where $v_i \in A_i \subseteq G\del E$. Since $a \in K_i$ and $v_i \notin K_i$, we have $v_i +a \notin K_i$, so $v_i + a \in L_i$. The element $u = a+v_i+x$ satisfies $u + a  \notin E$ and $u+v_i+a \notin E$, so $u \notin K_i \cup L_i$, giving $a+v_i+x \notin E$.
	
	Assume that $A_3$ is full, and let $v_3 \in A_3$. If $a+v_3$ and $a+v_3 + x$ are both not in $E$, then $\{v_3,a+v_3,a\}$ is a claw; by possibly replacing $v_3$ with $v_3 +x $ we may assume that $v_3 + a \in E$. Note that $v_1+v_3$ and $x+v_1+v_3$ are in the vacant set $A_2$; if $a+v_1+v_3 \in E$ then $\{x+v_3,a+v_3,a+v_1+v_3\}$ is a claw, so $a+v_1+v_3 \notin E$. 
	
	Let $w \in F \cap L_1$; we have $a+w \in K_1 + L_1$ so $a+w \notin E$.  Since $a+v_1 \in L_1$ we also have $a+w+v_1 \in E$. Now $w +v_3 \in A_3 \subseteq E$ while $w+v_1+v_3 \in A_2 \subseteq G\del E$; it follows that $\{w+v_3,a+v_3,a+w+v_1\}$ is a claw. This contradiction shows that $A_3$ is vacant. 
	
	Now assume that $|I| = 2$. Recall that $w \in F \cap L_1$ and $a+w \notin E$; it follows that $w \in L_2$ also, as $w \in L_2 \cup K_2$, and $w \in K_2$ would imply that $w+a \in K_2 \subseteq E$. For each $i \in \{1,2\}$ we have shown that $a+v_i \in L_i$, so $a + w + v_i \in L_i $, giving $a + w + v_i \in E$. 
	
	We have $a+w+v_1+v_2 \in E$, as otherwise $\{a,a+v_1,a+w+v_2\}$ is a claw. But now $\{a+w+v_1,a+w+v_2,a+w+v_1+v_2\}$ is a claw, giving a contradiction. Thus $|I| = 1$ as required. 
\end{subproof}


\begin{claim}
	$F$ has no bad vacant coset in $H$. 
\end{claim}
\begin{subproof}
	Let $A$ be such a coset. We first argue that $A$ is the only bad coset; indeed, if $A'$ is another bad coset then $(A,A')$ contradicts \ref{bothvacant} if $A'$ is vacant, the pair $(A',A+A')$ contradicts ~\ref{bothfull} if $A'$ and $A+A'$ are both full, and the pair $(A,A+A')$ contradicts \ref{bothvacant} if $A'$ is full and $A+A'$ is vacant. Thus $A$ is the only bad coset. 
	
	We now argue that the flat $F' = \cl(F \cup A \cup B) = F \cup A \cup B \cup (A+B)$ decomposes $M$. Since $F' \ne \varnothing$ while $\dim(F') = \dim(F) +2 < \dim(G)$, it suffices to show that $F'$ has no mixed cosets in $G$. Let $C$ be a coset of $F'$; we have 
	\[C = C_0 \cup (A+C_0) \cup (B+C_0) \cup (A+B+C_0)\]
	for some coset $C_0$ of $F$ in $H$. If exactly one of $C_0$ and $A+C_0$ is full, then either $(A,C_0)$ or $(A,A+C_0)$ contradicts \ref{bothvacant}. Thus $C_0$ and $A+C_0$ are either both empty or both full. Since they are both good, this implies that $C$ is either empty or full. So $F'$ is a decomposer of $M$, contrary to assumption. 
\end{subproof}

Thus, all bad cosets are full. 

\begin{claim}\label{existsvacant}
	$F$ has a vacant coset in $H$.
\end{claim}
\begin{subproof}
	Suppose not; we show that $M^c$ is triangle-free. Let $T$ be a triangle in $M^c$. Since $T \cap H$ is nonempty while $H \del F \subseteq E$, we have $T \cap F \ne \varnothing$. If $T \subseteq F \cup B$, let $F' = F \cup B \cup A \cup (A+B)$ for some bad coset $A$. Otherwise, let $F' = \cl(F \cup B \cup T)$. In either case, we have $T \subseteq F'$, and $F' = F \cup B \cup A \cup (A+B)$ for some coset $A$. 
	
	If $A$ is good, then by construction we have $T \not\subseteq F \cup B$, so $T$ intersects $A \cup (A+B)$. But $A$ is a full good coset, so $A \cup (A+B) \subseteq E$, contradicting the fact that $T \subseteq G\del E$. If $A$ is bad, then \ref{badstructure} implies that $(M|F')^c$ is triangle-free. This contradicts $T \subseteq F'$. 
\end{subproof}
Let $K$ be a maximal flat of $H$ that is disjoint from $F$, so each coset $A$ of $F$ in $H$ intersects $K$ in a unique element $v_A$, with $v_A \in E$ if and only if $A$ is full, while $A = F + v_A$. Let $P \subseteq K$ be the set of all $v_A$ for which $A$ is bad, let $Q$ be the set of all $v_A$ for which $A$ is good and full, and $R$ be the set of all $v_A$ for which $A$ is good and vacant. Now \ref{bothfull} implies that there is no triangle $T$ of $K$ with $|T \cap R| = 1$ and $|T \cap P| \ge 1$; by Lemma~\ref{PQR}, each coset of $\cl(P)$ in $K$ is contained in $Q$ or $R$. We have $R \ne \varnothing$ by \ref{existsvacant}, so $\cl(P) \ne K$. 

We now argue that the flat $F' = \cl(F \cup \{a\} \cup P)$ decomposes $M$. Clearly $F' \ne \varnothing$, and the fact that $\cl(P) \ne K$ implies that $F' \ne G$. Consider a coset $C$ of $F'$. We have $C = \vs{F'} + u = \vs{(F \cup (F + B))} + \vs{\cl(P)} + u$ for some $u \in G \del F'$; since $K$ contains a flat that is maximally disjoint from $F'$, we can take $u \in K  \del \cl(P)$.

The set $C_0 = \vs{\cl(P)} + u$ is a coset of $\cl(P)$ in $K$, so is contained in either $Q$ or $R$. Therefore $C \subseteq \vs{F \cup (F + B)} + Z$ for some $Z \in \{Q,R\}$. Now 

\[\vs{F \cup (F+B)} + Z = \bigcup_{z \in Z}\left((\vs{F} + z) \cup ((F + B) + z)\right).\]
 If $Z = Q$ then each coset $\vs{F} + z$ and $(F + B) + z$ is full by definition of $Q$, so $C$ is full. Similarly, if $Z = R$ then $\vs{F} + z$ and $(F+B)+z$ are empty for all $z$, so $C$ is empty. Thus $F'$ is a decomposer of $M$, yielding a final contradiction.
\end{proof}

At this point, we can easily reduce Theorem~\ref{main} to a finite computation, showing that it suffices to verify the result in dimension at most $8$. Although actually performing this check computationally would likely be impossible, we include the argument here for interest. 

\begin{theorem}\label{computation_reduce}
	If Theorem~\ref{main} holds for all matroids of dimension at most $8$, then it holds in general. 
\end{theorem}
\begin{proof}
	Let $M = (E,G)$ be a minimal counterexample to the theorem; we may assume that $\dim(M) \ge 9$. Since $M^c$ is not triangle-free, there is a triangle $T \subseteq G\del E$. Since $M$ is not a PG-sum, there is a plane $P$ for $G$ for which $M|P$ is not a PG-sum by Lemma~\ref{PG_characterisation}. Since $M \notin E_3$, there is a plane $Q$ of $G$ for which $|E \cap Q|$ is odd. 
	
	Note that $T \cup P \cup Q$ is contained in a flat $F$ of dimension at most $2+3+3 = 8$. Let $H$ be a hyperplane of $G$ containing $F$. The existence of $T,P$ and $Q$ certifies that $(M|H)^c$ is not triangle-free, and that $M|H$ is not even-plane or a PG-sum; since $M|H$ is not a counterexample, it follows that $M|H$ has a decomposer $F$. Theorem~\ref{hyperplane_decomposer} thus implies that $M$ has a decomposer, contrary to assumption. 
\end{proof}

\section{The main theorem}

In this section, we prove Theorem~\ref{main}. The strategy is an adaptation of the proof of Theorem~\ref{computation_reduce}, where we reduce the size of the base case from $8$ to something more manageable. The $8$ appears in the argument above because, naively, if a matroid is not in one of our three basic classes, then it contains a certificate of this fact in dimension at most $8$. Our argument below is essentially reducing the size of such a certificate in the important cases. 	

Recall that $P_5$ is the three-dimensional matroid with five elements. This matroid is of particular interest at this point since its presence in a matroid $M$ certifies that $M$ is neither a PG-sum nor even-plane. 

\begin{lemma}\label{basecase_shortcut}
	Let $M = (E,G)$ be a claw-free matroid and let $H$ be a hyperplane of $G$ for which $M|H$ is a strict PG-sum such that $M|H \notin \cE_3$ and $(M|H)^c$ is not triangle-free. Then either 
	\begin{itemize}
		\item $M$ has an induced $P_5$-restriction,
		\item $M$ is a strict PG-sum, or
		\item $M$ has a decomposer.
	\end{itemize}
\end{lemma}
\begin{proof}
	Suppose that none of the conclusions hold. Let $K_0,K_1$ be the disjoint nonempty flats of $H$ whose disjoint union is $E \cap H$. If one of the $K_i$ has dimension $1$, then since $M|H$ is strict, the other is a hyperplane of $H$, which meets every triangle of $H$; this contradicts the hypothesis that $H\del E$ contains a triangle. Therefore each $K_i$ has dimension at least $2$. If they both have dimension $2$ then $M|H \in \cE_3$; thus $K_0$ or $K_1$ has dimension at least $3$, and therefore $\dim(H) \ge 5$ and $\dim(G) \ge 6$.
	
	Since $M$ has no decomposer, there exists $a \in E \del H$. Let $X_1 = (a+E) \cap H$ and $X_0 = H \del X_1$. Think of the indices of $X_0$ and $X_1$ as belonging to $\bZ_2$.
	
	Since $M$ has no $P_5$-restriction, no plane $P$ of $G$ contains precisely five elements of $E$. For every triangle $T$ of $H$, the plane $\cl(T \cup \{a\})$ contains exactly $1 + |T \cap E| + |T \cap X_1|$ elements of $E$, so we have
	\begin{claim}\label{restrict_triangles}
		$|T\cap E| + |T \cap X_1| \ne 4$ for every triangle $T$ of $H$.
	\end{claim}
	
	For any triangle $T$ of $K_i$, it follows that $|T \cap X_1| \ne 1$; therefore
	\begin{claim}\label{subspaces}
		For each $i \in \{0,1\}$ the set $X_0 \cap K_i$ is a flat of $K_i$. 
	\end{claim}
	This also imposes structure on the elements of $K_0 + K_1$. 
	
	\begin{claim}\label{sumsets}
		For each $i,j \in \bZ_2$ we have $(X_i \cap K_0) + (X_j \cap K_1) \subseteq X_{1+i+j}$. 
	\end{claim}
	\begin{subproof}
		Let $x \in K_0 \cap X_i$ and $y \in K_1 \cap X_j$ and $x+y \in X_{\ell}$, where $\ell \in \bZ_2$. If $i = j = 0$ then, since $\{a,x,y\}$ is not a claw, we have $a+x+y \in E$ and so $\ell=1$. Otherwise the triangle $T = \{x,y,x+y\}$ satisfies $|T \cap E| = 2$ and $1 \le |T \cap X_1| \le 3$, so since $|T \cap E| + |T \cap X_1| \ne 4$ we have $|T \cap X_1|$ odd. It follows that $i+j+\ell = 1$ in $\bZ_2$ and so $\ell = 1+ i +j$, giving the claim.
	\end{subproof}
	\begin{claim}\label{polarise}
		For each $i \in \{0,1\}$,  either $K_i \subseteq X_0$ or $|K_i \cap X_0| = 1$.
	\end{claim}
	\begin{subproof}
		Suppose first that $K_i \subseteq X_1$. If $K_{1-i} \subseteq X_0$ then by \ref{sumsets} we have $K_0 + K_1 \subseteq X_0$, and it follows that $E$ is the union of the disjoint flats $\cl(K_i \cup \{a\})$ and $K_{1-i}$, so is a PG-sum, contrary to assumption. If there is some $v \in K_{1-i} \cap X_1$, then let $T = \{u_1,u_2,u_3\}$ be a triangle of $K_i$. We have $T + v \subseteq X_1$ by \ref{sumsets}, so now the triangle $T' = \{u_2,u_1+v,u_3+v\}$ satisfies $|T' \cap E| = 1$ and $T'\subseteq X_1$, contradicting ~\ref{restrict_triangles}. Therefore $X_0 \cap K_i$ is nonempty for both $i$. 
		
		If the conclusion fails, then there exist $z_1,z_2 \in K_i \cap X_0$ and $v \in K_i \cap X_1$; since $K_{i-1}\not\subseteq X_1$ there also exists $y \in K_{1-i}\cap X_0$. Since $X_0 \cap K_i$ is a flat we have $z_1 + z_2 \in X_0$ and $v+z_1+z_2 \in X_1$. By \ref{sumsets} we have $y+z_2 \in X_1$ and $y + v+z_1+z_2 \in X_0$. Now $\{z_1+v,a+y+z_2,a+v\}$ is a claw, a contradiction.
	\end{subproof}
	\begin{claim}
		$|K_i \cap X_0| = 1$ for both $i$.
	\end{claim}
	\begin{subproof}
		If not, then there is some $i \in \{0,1\}$ for which $K_i \subseteq X_0$ by ~\ref{polarise}. 
		If $K_{1-i} \subseteq X_0$, then \ref{sumsets} implies that $K_0 + K_1 \subseteq X_1$. Since the dimensions of $K_0$ and $K_1$ are at least $2$ and sum to at least $5$, we may assume by symmetry that $\dim(K_1) \ge 3$. Let $\{x_1,x_2\} \subseteq K_0$ and $\{y_1,y_2,y_3\} \subseteq K_{1}$ be linearly independent sets. Using $K_0 + K_1 \subseteq X_1$ together with $K_0 + K_1 \subseteq G\del E$, the set \[\{a+x_1+y_1+y_3,a+x_2+y_2+y_3,a+x_1+x_2+y_1+y_2+y_3\}\] is a claw, giving a contradiction.  
		
		Otherwise, the set $K_{1-i} \cap X_0$ has size $1$ by \ref{polarise}; let $w$ be its element. Now \ref{sumsets} gives $X_1 = (K_{1-i} \del \{w\}) \cup (K_i + w)$ and so 
		\begin{align*}
		E &= K_i \cup (K_i + a+w)  \cup K_{1-i} \cup (K_{1-i} + a+w)
		\end{align*}
		which implies that $E + a+w = E$, and so $\{a+w\}$ decomposes $M$, a contradiction. 
	\end{subproof}
	
	Assume now by symmetry that $\dim(K_1) \ge 3$; thus $K_1 \cap X_1$ contains a triangle $T = \{x,y,x+y\}$. Let $u \in K_0 \cap X_1$; now $u + T \subseteq X_1$ by \ref{sumsets}. Therefore the plane $\cl(\{a,x,y+w\})$ contains exactly five elements of $E$, giving a contradiction.
\end{proof}

We now restate and prove Theorem~\ref{main}.
\begin{theorem}
	If $M$ is a claw-free matroid, then either
	\begin{itemize}
		\item $M$ is even-plane, 
		\item $M^c$ is triangle-free,
		\item $M$ is a strict PG-sum, or
		\item $M$ has a decomposer.
	\end{itemize}
\end{theorem}
\begin{proof}
	Let $M = (E,G)$ be a counterexample of smallest possible dimension. Clearly $\dim(M) \ge 3$, since otherwise $M \in \cE_3$. It is easy to check the following:
	\begin{claim}\label{tiny_decomposer}
		Every $3$-dimensional, odd-sized claw-free matroid has a one-element decomposer. 
	\end{claim}
	This gives $\dim(M) \ge 4$. Also observe that for every hyperplane $H$ of $G$, the matroid $M|H$ has no decomposer, as otherwise we obtain a contradiction from Theorem~\ref{hyperplane_decomposer}. 
	
	Since $M \notin \cE_3$, there is a plane $P$ of $G$ for which $|P \cap E|$ is odd, and since $M^c$ is not triangle-free, there is a triangle $T \subseteq G \del E$. Choose $P$ and $T$ so that their intersection is as large as possible. 
	
	\begin{claim}
		$\dim(M) = 5$.
	\end{claim}
	\begin{subproof}
		Suppose not, so either $\dim(M) = 4$ or $\dim(M) \ge 6$. If $\dim(M) = 4$, then $P$ is a hyperplane of $G$, and \ref{tiny_decomposer} implies that $M|P$ has a decomposer, giving a contradiction.
		
		Therefore $\dim(M) \ge 6$. Let $H$ be a hyperplane of $G$ containing $P \cup T$. By construction the matroid $M|H$ is not even-plane, while $(M|H)^c$ contains a triangle; since $M|H$ has no decomposer but is not a counterexample, it is a strict PG-sum. Since $M$ is not a strict PG-sum and has no decomposer, Lemma~\ref{basecase_shortcut} implies that $M$ has a $P_5$-restriction $M|P'$. Let $H'$ be a hyperplane of $G$ containing $T$ and $P'$. Now the existence of $P'$ certifies that $M|H' \notin \cE_3$ and $M|H'$ is not a PG-sum, and $T$ certifies that $(M|H')^c$ is not triangle-free. Since $M|H'$ has no decomposer, this contradicts the minimality in the choice of $M$. 	
	\end{subproof}
	
	\begin{claim}
		$P \subseteq E$. 
	\end{claim}
	\begin{subproof}
		We first argue that $P \cap T$ is empty. If $P \cap T \ne \varnothing$ then $\dim(\cl(P \cup T)) \le 4$; let $H$ be a hyperplane containing $P \cup T$. Since $M|H$ is not a counterexample, but $M|H \notin \cE_3$ while $T \subseteq H \del E$ and $M|H$ has no decomposer, we conclude that $M|H$ is a strict PG-sum. 
		
		However, for each $4$-dimensional strict PG-sum, the ground set is either the disjoint union of a point and a hyperplane, or of two triangles. If $M|H$ has the former structure then every triangle of $H$ intersects the hyperplane, contradicting the existence of $T$. If $M|H$ has the latter structure, then every plane of $H$ has odd intersection with each of the two triangles so has even intersection with $E$; this contradicts the existence of $P$. 
		
		Therefore $P \cap T$ is empty. We now argue that $P \subseteq E$. If not, let $v \in P\del E$. If $u+v \notin E$ for some $u \in T$, then $T' = \{u,v,u+v\}$ is a triangle contained in $G\del E$ that intersects $P$ in more elements than $T$ does; this contradicts the choice of $T$ and $P$. Thus $v+T \subseteq E$. But this implies that $\cl(T \cup \{v\}) \cap E = v+T$ and so $M$ has a claw, a contradiction. Thus $P \subseteq E$. 
	\end{subproof}
	
	Since $T \subseteq G \del E$ and $P \subseteq E$, we have $T \cap P = \varnothing$ and so the fact that $\dim(G) = 5$ implies that $G = \cl(P \cup T)$.  Let $T = \{u_1,u_2,u_3\}$ and for each $u \in T$, let $A_u = P \cap (u+E)$. 
	
	\begin{claim}
		 $A_{u_1} \Delta A_{u_2} \Delta A_{u_3} = P$, and each of $A_{u_1},A_{u_2},A_{u_3}$ is either a triangle of $P$ or equal to $P$.
	\end{claim}
	\begin{subproof}
		If the first conclusion fails, then there is some $v \in P$ for which $|\{u \in T\colon v \in A_u\}|$ is even. Then the plane $Q = \cl(\{v\} \cup T)$ has odd intersection with $E$ and intersects $T$, contradicting the choice of $P$ and $T$.
		
		To see the second conclusion, it suffices to show that each triangle $T'$ of $P$ has odd intersection with $A_u$; indeed, if $|T' \cap A_u|$ is even, then the plane $P' = \cl(T' \cup \{u\})$ has odd intersection with $E$ and also intersects $T$; this contradicts the choice of $P$ and $T$.
	\end{subproof}
	
	If at least two of the sets $A_u$ are equal to $P$ (say the first two), then $A_{u_3} = P \Delta P \Delta P = P$ and so $E =  G \del T$ which implies that $T$ decomposes $M$, a contradiction. 
	
	If exactly one of the $A_u$ (say $A_{u_1}$) is equal to $P$, then $A_{u_2}$ and $A_{u_3}$ are triangles with symmetric difference $P \Delta P = \varnothing$, so $A_{u_2} = A_{u_3} = T'$ for some triangle $T'$ of $P$. It follows that \[E = (\vs{u_1} + P) \cup (\{u_2\} + T') \cup (\{u_3\} + T') = (\vs{u_1} + P) \cup (\{u_2,u_3\} + T').\] Since $u_1+[u_1] = [u_1]$ and $u_1 + \{u_2,u_3\} = \{u_2,u_3\}$, this implies that $u_1 + E = E$, and therefore $\{u_1\}$ decomposes $M$, a contradiction. 
	
	Finally, if all three $A_u$ are triangles, then since they have symmetric difference $P$, it is easy to see that they are exactly the three triangles $T_1,T_2,T_3$ through some element $z$ of $P$. Let $\{y_1,y_2,y_3\}$ be a triangle of $P$ not containing $z$ for which $y_i \in A_{u_i}$ for each $i \in \{1,2,3\}$, so by construction, if $x \in P$ and $i \in \{1,2,3\}$, then $u_i + x \in E$ if and only if $x \in \{z,y_i,z+y_i\}$. Using this, we see that $\{z+y_2,u_1+z+y_1,u_3+z\}$ is a claw.		 
\end{proof}

As discussed earlier, the above theorem, together with an inductive argument, implies Theorem~\ref{structure}. 

\section{Corollaries}

We now now prove the corollaries of Theorem~\ref{structure} that were discussed in the introduction. They are all easy consequences, even if writing down their proofs takes a little work. 

\subsection*{$\chi$-boundedness}

Recall that a class $\cM$ of matroids is $\chi$-bounded by a function $f$ if $\chi(M) \le f(\omega(M))$ for all $M \in \cM$, and that $\cM$ is $\chi$-bounded if it is $\chi$-bounded by some $f$. It was shown in [\ref{bkknp}] that the class $\cE_3$ is not $\chi$-bounded, and therefore neither is the class of claw-free matroids. In this section we show that for every $N \in \cE_3$, the class of claw-free, $N$-free matroids is $\chi$-bounded by some function $f$ that grows exponentially. 

Call a function $f\colon \bZ_{\ge 0} \to \bZ_{\ge 0}$ \emph{superadditive} if $f(x+y) \ge f(x)+f(y)$ for all $x,y > 0$. (It is technically convenient not to insist that $f(0) = 0$.)  

\begin{lemma}\label{lj_boundedness}
	Let $\cM$ be a class of matroids and let $\cM'$ be its closure under lift-joins. If $\cM$ is $\chi$-bounded, then so is $\cM'$. Moreover, if $\cM$ is $\chi$-bounded by a superadditive function $f$, then $\cM'$ is $\chi$-bounded by $f$. 
\end{lemma}
\begin{proof}
	Let $g$ be a function $\chi$-bounding $\cM$. Define $g'$ by $g'(k) = g(k)$ for $k \le 1$, and $g'(k) = \max(g(k),\max_{1\le i <k}(g'(i) + g'(k-i)))$ for all $k > 1$. By construction we have $g'(k) \ge g(k)$ for all $k \ge 1$, and $g'$ is superadditive. Since $g \le g'$, the class $\cM$ is $\chi$-bounded by $g'$. Note also that if $g$ is superadditive, then an inductive argument implies that $g = g'$. 
	
	To show the lemma, it therefore suffices to argue that $\cM'$ is $\chi$-bounded by $g'$. Suppose not, and let $M \in \cM'$ have minimal dimension with $\chi(M) > g'(\omega(M))$. If $M \in \cM$ we have a contradiction. Otherwise, $M$ is the lift-join of two matroids $M_1,M_2 \in \cM'$ of smaller dimension, and now Lemma~\ref{lj_parameters} gives
	\begin{align*}
		\chi(M) &= \chi(M_1) + \chi(M_2) \\
		&\le g'(\omega(M_1)) + g'(\omega(M_2))\\
		&\le g'(\omega(M_1) + \omega(M_2)) = g'(\omega(M)),
	\end{align*}
	a contradiction.
\end{proof}

We now need to argue that the class of complements of triangle-free matroids is $\chi$-bounded. If $M^c$ is triangle-free then $\chi(M)$ is equal to $\dim(M)$ or $\dim(M)-1$, so  this essentially amounts to showing that $\dim(M)$ is bounded by a function of $\omega(M)$. This is a special case of a Ramsey theorem for projective geometries; the following is a consequence of Corollary~2 of [\ref{gr}], rephrased in our language. 

\begin{theorem}\label{ramsey}
	For all $s,t \ge 0$ there is an integer $n(s,t)$ such that every $n$-dimensional matroid $M$ satisfies $\omega(M) \ge s$ or $\omega(M^c) \ge t$. 
\end{theorem}

Unfortunately, the techniques in [\ref{gr}] give an enormous value for $n$, even when $t = 2$, which is the special case we will need. In order to obtain a somewhat reasonable $\chi$-bounding function we use the following theorem of Sanders [\ref{sanders}] to derive a better bound for $n(s,2)$.

\begin{theorem}[{[\ref{sanders}, Theorem 4.1]}]\label{sandersthm}
	Let $G \cong \PG(n-1,2)$, let $X \subseteq \vs{G}$, and let $\alpha = |X|/2^n$. If $\alpha \le \tfrac{1}{2}$, then $X+X$ contains a subspace of $\vs{G}$ of dimension $n -\left\lceil n/{\log_2\left(\frac{2-2\alpha}{1-2\alpha}\right)}\right\rceil$.
\end{theorem} 

\begin{corollary}\label{off_diagonal}
	Let $s \ge 2$ be an integer. If $M = (E,G)$ is a matroid for which $M^c$ is triangle-free and $\dim(M) \ge 2^s(s+1)$, then $\omega(M) \ge s$. 
\end{corollary}
\begin{proof}
	Let $\dim(M) = n$. Suppose that $\omega(M) < s$; Theorem~\ref{bbt} implies that $|E| \le 2^n-2^{n-s}$. If equality holds then $M$ is an order-$s$ Bose-Burton geometry so $G\del E$ is a flat of dimension $n - s \ge 2$ and thus contains a triangle, contrary to hypothesis. Therefore $|E| \le 2^n-2^{n-s}-1$. 
	
	Let $E^c = G\del E$. We have $|E^c| = 2^n-1-|E| \ge 2^{n-s}$. Let $X$ be a $2^{n-s}$-element subset of $E^c$; we have $\alpha = |X|/2^n = 2^{-s} \le \tfrac{1}{2}$. Now 
	\[\log_2\left(\tfrac{2-2\alpha}{1-2\alpha}\right) = 1 + \tfrac{1}{\ln(2)} \ln(1+ \tfrac{1}{1/\alpha-2}) = 1 + \tfrac{1}{\ln(2)}\ln(1+ \tfrac{1}{2^s-2})\]
	Using $\ln(1+x) \ge x - \tfrac{1}{2}x^2$ and the fact that $\tfrac{1}{2^s-2} - \tfrac{1}{2(2^s-2)^2} \ge 2^{-s}$ for $s \ge 2$, this gives $\log_2\left(\tfrac{2-2\alpha}{1-2\alpha}\right) \ge 1+\tfrac{1}{\ln(2)}2^{-s}$. By Theorem~\ref{sandersthm}, the set $X+X$ thus contains a subspace $F$ of dimension $n - \lceil n/ (1+\tfrac{1}{2^s \ln 2}) \rceil$. Using $\lceil x \rceil \le x+1$ we get
	\begin{align*}
	\dim(F) \ge n - \lceil n/ (1+\tfrac{1}{2^s \ln 2}) \rceil \ge n-\tfrac{n}{2^s \ln 2 + 1} -1 \ge \tfrac{n}{2^s} - 1 \ge s,
	\end{align*}
	where we use $2^s \ln 2 +1 \le 2^s$. But now $F \subseteq E^c + E^c$, so since $E^c$ contains no triangle, we have $F \subseteq E$. This contradicts the fact that $\omega(M) < s$. 
\end{proof}

Rephrased in the language of Theorem~\ref{ramsey}, this result states that $n(s,2) \le (s+1)2^s$. It would be interesting to know whether a subexponential bound is possible; certainly it is for the analogous problem in graph theory. We now have enough to prove Theorem~\ref{chi_bounded_simple}. We restate it here with an explicit $f$ that grows exponentially. 

\begin{theorem}\label{chi_bounded_main}
	Let $N \in \cE_3$. The class of $N$-free, claw-free matroids is $\chi$-bounded by the function 
	$f(k) = (k+2)2^{k+1} + k(\dim(N)+4)$.
\end{theorem}
\begin{proof}
	Let $\cM_1$ denote the class of $N$-free matroids in $\cE_3$, let $\cM_2$ denote the class of matroids whose complement is triangle-free, and let $\cM_3$ denote the class of PG-sums. Let $\cM' = \cM_1 \cup \cM_2 \cup \cM_3$.  By Theorem~\ref{structure}, every claw-free matroid $M$ is obtained via lift-joins from `basic' matroids in $\cE_3$, $\cM_2$ or $\cM_3$, and since all the basic matroids are induced restrictions of $M$, if $M$ is $N$-free, then so are all the basic matroids. Thus $M$ lies in the closure under lift-joins of $\cM'$. 
	
	The function $h(k)= (k+2)2^{k+1}$ is superadditive since 
	\[(x+y+2)2^{x+y+1} \ge (x+y+4)2^{x+y} \ge (x+2)2^{x+1} + (y+2)2^{y+1}\] for all $x,y \ge 1$. Therefore $f$ is superadditive as it is the sum of two superadditive functions. By Lemma~\ref{lj_boundedness} it is thus enough to show that $\cM'$ is $\chi$-bounded by $f$. Indeed, if $M \in \cM_1$ then $\chi(M) \le \dim(N) + 4 \le f(\omega(M))$ by Theorem~\ref{even_plane_bounded_cn}. If $M \in \cM_2$ with $\omega(M) = s$, then $\chi(M) \le \dim(M) < 2^{s+1}(s+2) \le f(s)$ by Corollary~\ref{off_diagonal}. Finally, if $M \in \cM_3$ then $\chi(M) = \omega(M) \le f(\omega(M))$ by Lemma~\ref{pg_sum_chi}; the theorem follows. 
\end{proof}

This theorem is enough to characterise exactly which down-closed classes of claw-free matroids are $\chi$-bounded.

\begin{corollary}
	If $\cM$ is a class of claw-free matroids that is closed under taking induced restrictions, then $\cM$ is $\chi$-bounded if and only if $\cE_3 \not\subseteq \cM$. 
\end{corollary}
\begin{proof}
	If $\cE_3 \subseteq \cM$, then Theorem~\ref{e3_unbounded} implies that $\cM$ is not $\chi$-bounded. Otherwise, there is some $N \in \cE_3$ with $N \notin \cM$, so all matroids in $\cM$ are $N$-free. Theorem~\ref{chi_bounded_main} now gives the result.
\end{proof}

\subsection*{Rough structure} 
We now restate and prove Theorems~\ref{roughclaw}, ~\ref{flatroughomega} and ~\ref{flatroughalpha}. 

\begin{theorem}
	For all $s \ge 1$, there exists $k \ge 2$ such that, if $M$ is a claw-free matroid with $\omega(M) \le s$, then $M$ is a lift-join of matroids $M_1, \dotsc, M_{2s+1}$, such that each $M_i$ either has dimension at most $k$, or is even-plane. 
\end{theorem}
\begin{proof}
	Let $k = k(s)= 2^{s+1}(s+2)$ for each $s$. Recall that Lemma~\ref{lsassoc} shows that $\ls$ is associative. By Theorem~\ref{structure}, there is some $t$ for which there are matroids $M_1,\dotsc,M_t$ with $M = \ls_{i=1}^t M_i$, for which each $M_i$ is either an even-plane matroid, the complement of a triangle-free matroid, or a strict PG-sum. By including dimension-zero matroids, we may assume that $t \ge 2s+1$; choose $t$ so that if $t > 2s+1$, then $t$ is as small as possible. 
	
	Each $M_i$ that is not even-plane is either a strict PG-sum, or the complement of a triangle-free matroid. It is clear that a strict PG-sum with $\omega \le s$ has dimension at most $2s$. If $M_i^c$ is triangle-free and $\omega(M_i) \le s$, then by Corollary~\ref{off_diagonal} we have $\dim(M_i) < 2^{s+1}(s+2)$; in either case, $\dim(M_i) \le k$. 
	
	It remains to argue that $t = 2s+1$; suppose not. Using $\omega(N \ls N') = \omega(N) + \omega(N')$, it is routine to show by induction that $\omega(M) \ge c$, where $c$ is the number of $M_i$ that have a nonempty ground set. Since $\omega(M) \le s$, this implies that $c \le s < \tfrac{1}{2}(t-1)$. Therefore there are two consecutive matroids $M_i,M_{i+1}$ that are both empty. But if this is the case, then one could replace $M_i,M_{i+1}$ with the empty (and thus even-plane) matroid $M_i \ls M_{i+1}$ in the sequence to shorten its length, contradicting the minimality of $t$. 
\end{proof}

\begin{theorem}
	For all $s \ge 1$ there exists $k \ge 2$ such that, for every claw-free matroid $M = (E,G)$ with $\omega(M) \le s$, there is a flat $F$ of $G$ whose codimension at most $k$, such that $M|F$ is the lift-join of $2s+1$ even-plane matroids. 
\end{theorem}
\begin{proof}
	Let $s \ge 1$. Let $k'$ be the constant depending on $s$ given by the previous theorem, and let $k = (2s+1)k'$. 
	By the previous theorem we have $M = \ls_{i=1}^{2s+1} M_i$ where each $M_i = (E_i,G_i)$ either has dimension at most $k'$ or is even-plane. For each $1 \le i \le 2s+1$, let $F_i = G_i$ if $M_i$ is even-plane, and let $F_i = \varnothing$ otherwise. Let $F = \cl(\cup_i F_i)$. By the second part of Lemma~\ref{lj_parameters} the matroid $M|F$ is the lift-joint of $2s+1$ even-plane matroids $M_i|F_i$, and since $\dim(F) = \sum_{i=1}^{2s+1}\dim(F_i) \ge \sum_{i=1}^{2s+1} (\dim(G_i) - k') \ge \dim(M) - k$, the result follows. 
\end{proof}

\begin{theorem}
For all $s \ge 1$ there exists $k \ge 2$ such that, for every claw-free matroid $M = (E,G)$ with $\alpha(M) \le s$, there is a flat $F$ of $G$ whose codimension at most $k$, such that $M|F$ is the lift-join of $2s+1$ matroids whose complements are triangle-free. 
\end{theorem}
\begin{proof}
	By Theorem~\ref{structure}, we have $M = \ls_{i=1}^t M_i$, where each $M_i = (E_i,G_i)$ is either the complement of a triangle-free matroid, a strict PG-sum, or even-plane. Let $n_i = \dim(G_i)$ for each $i$. Let $\ell$ be the number of $M_i$ for which $\alpha(M_i) \ge 2$ (i.e. $M_i^c$ is not triangle-free). Using Lemma~\ref{lj_parameters} we have $s \ge \alpha(M) = \sum_{i=1}^t \alpha(M_i) \ge 2\ell$. For each $i$ we produce a flat $F_i$ of $G_i$ so that $(M|F_i)^c$ is triangle-free, and so that the sum of the codimensions of the $F_i$ is bounded. 
	
	If $\alpha(M_i) \leq 1$, then set $F_i = G_i$. Otherwise $M_i$ is even-plane or a strict PG-sum. In either case, since $n_i - \chi(M_i) = \alpha(M_i) \le s$, we have $n_i \le \chi(M_i) + s$ for each $i$. If $M_i \in \cE_3$, then let $F_i = \varnothing$; since $M_i$ does not contain the empty $s$-dimensional matroid $O_s$, Theorem~\ref{even_plane_bounded_cn} gives $\chi(M_i) \le 4$, and so $n_i \le s+4$ and thus $F_i$ has codimension at most $s+4$. If $M_i$ is a strict PG-sum, then let $F_i$ be a larger of the two flats whose union is $E_i$. Lemma~\ref{pg_sum_chi} gives $s \ge \alpha(M_i) = n_i - \chi(M_i) = n_i - \omega(M_i) = n_i - \dim(F_i)$, so $F_i$ has codimension at most $s$ in $G_i$. Since $F_i \subseteq E_i$, we also have $(M|F_i)^c$ triangle-free. 
	
	Now let $F = \cl(\cup_{i=1}^t F_i)$. For each $M_i$ for which $\alpha(M) \le 1$ we have $\dim(F_i) = n_i$, and for each other $M_i$ we have $\dim(F_i) \ge n_i - (s+4)$. There are at most $\ell \le s/2$ matroids $M_i$ of the second type; it follows that $\dim(F) \ge n - \tfrac{s}{2}(s+4)$. Moreover, by construction each $M|F_i$ is the complement of a triangle-free matroid, and we have $M|F = \ls_{i=1}^t(M|F_i)$. 
	
	It remains to show that we can choose $t \le 2s+1$ (we can assume that $t \ge 2s+1$ by including dimension-zero matroids if necessary). Since $s \ge \alpha(M|F) = \sum_{i=1}^t \alpha(M|F_i)$, there are at most $s$ different $F_i$ for which $\alpha(M|F_i) > 0$; if $t > 2s+1$ then there thus exists $i$ for which $\alpha(M|F_i) = \alpha(M|F_{i+1}) = 0$; i.e. these two $M|F_i$ are both projective geometries. But then $M_i' = ((M|F_i) \ls (M|F_{i+1}))$ is also a projective geometry; replacing these two elements of the sequence by just $M_i'$ would give a shorter sequence of matroids whose lift-join is $M|F$. By repeating this operation as required, we conclude that $M|F$ is the lift-join of exactly $2s+1$ matroids whose complements are triangle-free, as required. 
\end{proof}

\subsection*{Density}
We restate and prove Theorem~\ref{density}. The proof is a little tedious but routine; it amounts to proving that PG-sums where the two flats have almost equal size are sparser than even-plane matroids, the complements of triangle-free matroids, or anything constructed by their lift-joins. 
\begin{theorem}
	Let $M = (E,G)$ be a full-rank, $r$-dimensional claw-free matroid. Then $|E| \ge 2^{\lfloor r/2 \rfloor} + 2^{\lceil r/2 \rceil} - 2$. Equality holds precisely when either $M \cong C_4$, or $E$ is the disjoint union of two flats of dimensions $\lfloor r/2 \rfloor$ and $\lceil r/2 \rceil$. 
\end{theorem}
\begin{proof}
	The theorem is easy to check when $r \le 3$, since full-rank $r$-dimensional matroids have at least $r$ elements, and claw-free ones have at least $4$ elements when $r = 3$. We prove the theorem for $r > 3$. 
	
	Define $f\colon \bZ \to \bZ$ by $f(n) = 2^{\lfloor n/2 \rfloor} + 2^{\lceil n/2 \rceil} -2$.
	It is routine to check that $f(n+1) \le 2f(n)$ for all $n \ge 1$, with equality if and only if $n \in \{1,2\}$; it follows by induction that $f(n+i) \le 2^i f(n)$ for all $n,i \ge 1$, with equality only if $n+i \le 3$. From this, we get
	$f(n) < 2^{n-1}$ for all $n \ge 4$, and that $f(n) < 2^{n-2}$ for all $n \ge 6$.
	
	If $M = (E,G)$ is a rank-$r$ strict PG-sum, then $E$ is the disjoint union of flats of dimensions $d_1$ and $d_2$ with $d_1 + d_2 = r$. Thus $|E| = 2^{d_1} + 2^{d_2}- 2$, which is minimized precisely when $\{d_1,d_2\} = \{\floor{r/2},\ceil{r/2}\}$; thus $M$ satisfies the theorem.
	By Theorem~\ref{bbt}, the complement of an $r$-dimensional triangle-free matroid has at least $2^r-(2^r-2^{r-1}) = 2^{r-1}$ elements, so if $M^c$ is triangle-free, then $M$ satisfies the theorem. We now show that the same is true for even-plane matroids and for lift-joins of smaller matroids satisfying the theorem.
	\begin{claim}
		The theorem holds for even-plane matroids. 
	\end{claim}
	\begin{subproof}
		We first show that if $M = (E,G)$ is a full-rank, $r$-dimensional even-plane matroid with an induced $C_4$-restriction, then $|E| \ge 4(r-2)$. Suppose that $M|P \cong C_4$ for a plane $P$ of $M$. Since $M$ is full-rank, there must be at least $r-3$ cosets of $P$ that intersect $E$. We show that each such coset contains at least four elements of $E$, from which it will follow that $|E| \ge 4(r-3) + 4 = 4(r-2)$ as claimed. Consider such a coset $a + \vs{P}$, where $a \in E \del P$. If $a+P \subseteq E$ then $|E \cap (a + \vs{P})| = 8 > 4$. Otherwise there exists $x \in P$ for which $x + a \notin E$. For each triangle $T = \{x,y,z\}$ of $P$ containing $x$, the fact that $M|P \cong C_4$ implies that $|T \cap E|$ is even, so, since the plane $\cl(T \cup \{a\})$ also has even intersection with $E$, the set $E$ must contain exactly one of $y+a$ and $z+a$. Making this argument for each of the three triangles of $P$ containing $x$ implies that $E$ contains exactly three elements of $(a+\vs{P}) \del \{a\}$, so $|E \cap (a + \vs{P})| = 4$ as required, giving $|E| \ge 4(r-2)$.
		
		In particular, if $\dim(M) \in \{4,5\}$ and $M$ has an induced $C_4$-restriction, then since $8 > f(4)$ and $12 > f(5)$, we have $|E| \ge 4(r-2) >  f(r)$, so $M$ satisfies the theorem. 
		
		Now let $M = (E,G) \in \cE_3$ with $\dim(M) \ge 4$. If every triangle $T$ of $G$ has even intersection with $E$, then since $M$ is nonempty, we have $E = G \del H$ for some hyperplane, so $|E| = 2^{r-1} > f(r)$. Otherwise, there is a triangle $T$ of $G$ for which $|T \cap E|$ is odd, and since each of the $2^{r-2}-1$ planes of $G$ containing $T$ has even intersection with $E$, we have $|E| \ge |E \cap T| + 2^{r-2}-1 \ge 2^{r-2}$. If $\dim(G) \ge 6$ then since $f(r) < 2^{r-2}$, this implies the result.
		
		If $\dim(G) \in \{4,5\}$ then, as established, we may assume $M$ to have no induced $C_4$-restriction. If $M$ is a PG-sum then the result holds. Otherwise, by Lemma~\ref{PG_characterisation} and the fact that $M \in \cE_3$, there is a plane $P$ of $G$ for which $M|P \cong K_4$. Let $T'$ be a triangle of $P$ with $T' \subseteq E$. Since each plane $P'$ containing $T'$ has even intersection with $E$, this gives $|E| \ge |E \cap P| + (2^{r-2}-2) = 2^{r-2} + 4$, since $2^{r-2}-2$ is the number of planes other than $P$ that contain $T$. Since $r \in \{4,5\}$ we have $2^{r-2} + 4 > f(r)$, giving the result.
	\end{subproof}
	
	\begin{claim}
		If $M = M_1 \ls M_2$ is a full-rank matroid and $M_1,M_2$ satisfy the theorem, then $M$ satisfies the theorem.
	\end{claim}
	\begin{subproof}
		We may assume that $M_1$ and $M_2$ have positive dimension. Let $M_i = (E_i,G_i)$ and $d_i = \dim(G_i) \ge 1$, so $\dim(M) = d_1 + d_2$, and $E = E_1 \cup (\vs{G_1} + E_2)$. If $E_2$ is contained in a hyperplane $H_2$ of $G_2$, then $E \subseteq \vs{G_1} + \vs{H_2}$ which implies that $M$ is not full-rank; thus $M_2$ is full-rank, and so $|E_2| \ge f(d_2)$, which gives 
		\[|E| \ge |\vs{G_1}|\cdot |E_2| \ge 2^{d_1}f(d_2) > f(d_1+d_2) = f(\dim(M)),\] where we use $4 \le \dim(M) = d_1+d_2$ and the fact established earlier that $f(n+i) < 2^if(n)$ for $n+i > 3$.
	\end{subproof}

	The theorem now follows from an inductive argument using the claims above and Theorem~\ref{main}. 
\end{proof}

\subsection*{Excluding anticlaws}
Finally, we restate and prove Theorem~\ref{anticlaw}. Recall that a anticlaw is the complement of a claw, and that a matroid $M = (E,G)$ is called a target if there are distinct (possibly empty) flats $F_0 \subseteq \dotsc \subseteq F_k \subseteq G$ such that $E$ is the union of $F_{i+1} \del F_i$ over all even $i$. 

	\begin{theorem}\label{anticlaw_restated}
		A matroid $M$ is claw-free and anticlaw-free if and only if $M$ is a target.
	\end{theorem}
	
To prove this theorem, we begin with a straightforward lemma that describes the structure of claw-free, triangle-free matroids. The lemma and proof can also be found in [\ref{bkknp}], but the we repeat it here for completeness. 

	\begin{lemma}[{[\ref{bkknp}], Corollary 5.2}]\label{claw_triangle_free}
	A full-rank matroid $M = (E,G)$ is claw-free and triangle-free if and only if $M$ is an order-$1$ Bose-Burton geometry.
	\end{lemma}

\begin{proof}
We may assume that $\dim(M) \ge 2$. If $M$ is an order-$1$ Bose-Burton geometry, then it is clear that it is claw-free and triangle-free. For the converse, we make the following observation. Given three distinct elements $u,v,w \in E$, since $\{u,v,w\}$ is not a triangle, and $M$ is claw-free, there exists some $x \in \cl(\{u,v,w\}) \cap E$. Since $M$ is triangle-free, $x \notin \{u+v, v+w, u+w\}$, and therefore it follows that $x = u+v+w$. Thus, for all distinct $u,v,w \in E$ we have $u+v+w \in E$. Now, let $v_0 \in E$, and consider the set $E' = E + v_0$. We claim that $E'$ is a subspace of $\bF_2^n$. Clearly $0 \in E'$, and for any two distinct nonzero $x,y \in E'$, we have $x+y = ((x-v_0)+(y-v_0)+v_0) + v_0 \in E'$ by the above observation applied to $\{x-v_0, y-v_0, v_0\}$. Therefore $E'$ is a subspace of $\bF_2^n$. If $v_0 \in E'$ then $E = E'$ and so $M$ contains a triangle, giving the result or a contradiction. If $v_0 \notin E'$ then $E$ is a coset of $E'$; since $M$ is full-rank, it follows that $M$ is an order-$1$ Bose-Burton geometry.
\end{proof}

We are now ready to prove Theorem~\ref{anticlaw_restated}.  	
	
\begin{proof}[Proof of Theorem~\ref{anticlaw_restated}]
It is easy to see that a claw is not a target. By Lemma~\ref{targetsnice}, targets are closed under complementation and induced restrictions; it follows that targets are claw-free and anticlaw-free. It remains to show that claw-free, anticlaw-free matroids are targets; suppose otherwise and let $M = (E,G)$ be a counterexample of smallest possible dimension. 

By minimality, $M$ is full-rank, and since $M$ is claw-free and anticlaw-free, both $E$ and $G \del E$ contain triangles, as otherwise we can apply Lemma~\ref{claw_triangle_free} to $M$ or $M^c$ to conclude that $M$ is a target, giving a contradiction. By Theorem~\ref{main} and the fact that $G \del E$ contains a triangle, $M$ is either even-plane, a strict PG-sum, or the lift-join of two matroids of smaller dimension. Let $T$ be a triangle contained in $E$. 

Suppose first that $M$ is even-plane. Let $\cP$ be the collection of planes of $G$ containing $T$. For each $P \in \cP$, we have $|E \cap P| \in \{4,6\}$, and since $E \cap P$ contains a triangle but $M|P$ is not an anticlaw, we have $|E \cap P| = 6$ and $|E \cap (P \del T)| = 3$. Therefore \[|E| = |T| + \sum_{P \in \cP}|E \cap (P \del T)| = 3(1+|\cP|) = 3 \cdot 2^{\dim(M)-2}.\] But $E$ contains no plane of $G$, so  
Theorem~\ref{bbt} implies that $M$ is a Bose-Burton geometry of order $2$ and is thus a target, giving a contradiction.

Suppose that $M$ is a strict PG-sum, so $E$ is the disjoint union of two nonempty flats $F_1$ and $F_2$. One of these flats, say $F_1$, must contain $T$, but then $M \rvert \cl(T \cup \{v\})$ is a anticlaw for each $v \in F_2$, a contradiction.

Finally, suppose that $M = M_1 \ls M_2$ for matroids $M_1,M_2$ of smaller dimension that $M$. By minimality, both $M_1$ and $M_2$ are targets; by Lemma~\ref{targetsnice}, it follows that $M$ is a target, giving a contradiction.
\end{proof}

\section*{References}
\newcounter{refs}
\begin{list}{[\arabic{refs}]}
{\usecounter{refs}\setlength{\leftmargin}{10mm}\setlength{\itemsep}{0mm}}

\item\label{bb}
R. C. Bose, R. C. Burton, 
A characterization of flat spaces in a finite geometry and the uniqueness of the Hamming and the MacDonald codes, 
J. Combin. Theory 1 (1966), 96--104. 

\item\label{bkknp}
M. Bonamy, F. Kardo\v{s}, T. Kelly, P. Nelson, L. Postle,
The structure of binary matroids with no induced claw or Fano plane restriction

\item\label{cp}
M. Chudnovsky, P. Seymour, 
The structure of claw-free graphs, Surveys in Combinatorics 2005, 
London Math Soc Leture Note Series Vol. 324, 153--172. 

\item\label{fl}
J. Fox, L. Lov\'asz, 
A tight bound for Green's arithmetic triangle removal lemma in vector spaces,
arXiv:1606.01230 [math.CO].

\item\label{es}
J. Geelen, P. Nelson, 
An analogue of the Erd\H{o}s-Stone theorem for finite geometries
Combinatorica 35 (2015), 209--214.

\item\label{crit_threshold}
J. Geelen, P. Nelson,
The critical number of dense triangle-free binary matroids,
J. Combin. Theory Ser. B 116 (2016), 238--249.

\item\label{gr}
R. L. Graham, B. L. Rothschild,
Ramsey's theorem for $n$-parameter sets,
Proc. Amer. Math. Soc. 36 (1972), 341-346.

\item\label{green}
B. Green,
A Szemer\'{e}di-type regularity lemma in abelian groups, with applications,
Geometric \& Functional Analysis GAFA 15 (2005), 340--376.

\item\label{g85}A. Gy\'arf\'as, Problems from the world surrounding perfect graphs, Proceedings of the International Conference on Combinatorial Analysis and its Applications, (Pokrzywna, 1985), Zastos. Mat. 19 (1987), 413--441.

\item \label{oxley}
J. G. Oxley, 
Matroid Theory,
Oxford University Press, New York (2011).

\item\label{sanders}
T. Sanders, 
Green's sumset problem at density one half,
Acta Arithmetica 146 (2011), 91--101.

\item\label{s81}
D.P. Sumner, 
Subtrees of a graph and chromatic number, in The Theory and Applications of Graphs, (G. Chartrand, ed.), John Wiley \& Sons, New York (1981), 557--576.

\end{list}

\end{document}